\documentclass[10pt]{article}
\usepackage{amssymb,amsfonts}
\usepackage{pb-diagram}
\usepackage{amsmath,amsthm,amsfonts,amssymb}
\usepackage[usenames]{color}
\usepackage{hyperref}
\usepackage{soul}
\usepackage{graphicx}
\newtheorem{theorem}{Theorem}
\newtheorem{definition}{Definition}

\newtheorem{lemma}{Lemma}
\newtheorem{cor}{Corollary}
\newtheorem{remark}{Remark}

\newtheorem{prop}{Proposition}

\newcommand{\be}{\begin{enumerate}}
\newcommand{\ee}{\end{enumerate}}
\newcommand{\beq}{\begin{equation}}
\newcommand{\eeq}{\end{equation}}
\usepackage{amssymb}
\usepackage{amsmath}
\usepackage{amssymb}

\newcommand{\sst}{\; | \;}  
\newcommand{\gst}{\, | \,}  
\newcommand{\naturals}{\ensuremath{\mathbb{N}}} 
\newcommand{\cs}{,\;} 
\newcommand{\qs}{\:} 
\newcommand{\GammaPresentation}{\langle A\,|\,\mathcal{R}\rangle} 
\newcommand{\GPresentation}{\ensuremath{\langle Z\gst S\rangle}}
\newcommand{\Hom}{\mathrm{Hom}}
\newcommand{\wl}[1]{|{#1}|} 

\newenvironment{romanenumerate}
	{\begin{enumerate}

	}
        {
        
        \end{enumerate}
        } 
\newenvironment{arabicenumerate}
	{\begin{enumerate}

	}
        {
        
        \end{enumerate}
        } 

\newcommand{\braced}[4]{\left\{\begin{array}{ll} {#1} & {#2} \\ {#3} & {#4} \end{array}\right.}
\newcommand{\bracedTwo}{\braced}

\newcommand{\comment}[1]{} 

\begin{document}

\title{Decidability of the Elementary Theory of a Torsion-Free Hyperbolic Group}
\author{Olga Kharlampovich, Alexei Myasnikov}

\maketitle
\begin{abstract}
Let $\Gamma$ be a  torsion free  hyperbolic group. We prove that  the elementary theory of $\Gamma$  is decidable and admits an effective quantifier elimination to boolean combinations of $\forall\exists$-formulas. The existence of such quantifier elimination was previously proved in \cite{Sela}.
\end{abstract}

\section{Introduction}

 It was proved in \cite{Sela} that every first order formula in the theory of a torsion free hyperbolic group is equivalent to a boolean combination of $\forall\exists$-formulas. We will prove the following result.
\begin{theorem} \label{main} Let $\Gamma $ be a torsion free hyperbolic group. There exists an algorithm, given a first-order formula $\phi$,  to find a boolean combination of $\forall\exists$-formulas that define the same set as $\phi$ over $\Gamma$.\end{theorem}
This theorem will be proved in Section \ref{sec:6}. We will also prove  the following result in Section \ref{sec:4}.
\begin{theorem} \label{aetheory}The $\forall\exists$-theory of a torsion-free hyperbolic group is decidable. \end{theorem} 
These results imply
\begin{theorem} The elementary theory of a torsion-free hyperbolic group is decidable.
\end{theorem}
Notice that an algorithm to solve systems of equations in torsion free hyperbolic groups was constructed in \cite{RS95}. The problem of solving equations in such groups was reduced to the problem of solving equations in a free group, and Makanin's algorithm was used for solving equations in a free group \cite{Mak}.  Decidability of the existential theory of a free group was shown in \cite{Mak}, and decidability of the existential theory of a torsion free hyperbolic group was proved in \cite{Sela}, \cite{D1}, and, later, in  \cite{KMac}. It was shown in \cite{KMel} that the elementary theory of a free group is decidable (solution of an old problem of Tarski).  In \cite{KMel1} we proved the statement of Theorem \ref{main} for a free group.  

The techniques that we are using can be applied in some other important classes of groups, for example, partially commutative groups (right-angled Artin groups). It is known that the compatibility problem for systems of equations over partially commutative groups is decidable, see \cite{[DM06]}. Moreover, the universal (existential) and positive theories of partially commutative groups are also decidable, see \cite{[DL04]} and \cite{[CK07]}. An effective description of the solution set of systems of equations over a partially commutative group was 
given in \cite{CRK1} (using an analogue of Makanin-Razborov diagrams), and this description can be applied to study elementary theories of these groups. 

In this version of the paper  we corrected some errors found by D. Groves and H. Wilton \cite{GW1} in the previous version. We also corrected in Section \ref{ImFT} some similar errors in Theorem 2.3 in Sela's paper \cite{Sela} because we are using this result.

\section{Preliminary facts}
\subsection{Toral relatively hyperbolic groups}
A finitely generated group $G$ that is hyperbolic relative to a collection $\{P_{1},\ldots,P_{k}\}$ of subgroups is called \emph{toral} if $P_{1},\ldots,P_{k}$ are all abelian and $G$ is
torsion-free.

Many algorithmic problems in
(toral) relatively hyperbolic groups are decidable, and in particular we take note of the following for later use.
\begin{lemma}\label{Lem:AlgorithmsRelativelyHyperbolic}
In every toral relatively hyperbolic group $G$, the following hold.
\begin{arabicenumerate}
\item The conjugacy problem in $G$, and hence the word problem, is decidable.
\item If $g\in G$ is a hyperbolic element (i.e. not conjugate to any element of any $P_{i}$), then the centralizer $C(g)$ of $g$ is an infinite cyclic group.
Further, a generator for $C(g)$ can be effectively constructed.
\end{arabicenumerate}
\end{lemma}
\begin{proof}
The word problem was solved in \cite{Far98} and the conjugacy problem in \cite{Bum04}.
For the second statement, let $G=\langle A\rangle$ and let $g\in G$ be a hyperbolic element.
Theorem~4.3 of \cite{Osi06IJAC} shows that the subgroup
\[
E(g)= \{ h\in G \sst \exists\qs n\in\naturals :\qs h^{-1}g^{n}h=g^{\pm n}\}
\]
has a cyclic subgroup of finite index.  Since $G$ is torsion-free, $E(g)$ must be infinite cyclic (see for example the proof of
Proposition~12 of \cite{MR96}).  Clearly $C(g)\leq E(g)$, hence $C(g)$ is infinite cyclic.

To construct a generator for $C(g)$, consider the following results of D. Osin (see the proof of Theorem~5.17 and Lemma~5.16 in \cite{osin}) :
\begin{romanenumerate}
\item there exists a computable constant $N$, which depends on $G$ and the word length $\wl{g}$ \label{Osin2}
 such that if $g=f^{n}$ for some $f\in G$ and  positive $n$, then $n\leq N$;
\item  there is a computable function $\beta:\naturals \rightarrow\naturals$  \label{Osin1}
such that if $f$ is an element of $G$ with $f^{n}=g$ for some positive $n$, then $f$ is conjugate to some element $f_{0}$ satisfying $\wl{f_{0}}\leq \beta(\wl{g})$.
\end{romanenumerate}
We proceed as follows.  Let
$\mathcal{F}$ be the set of all $f\in G$ such that $\wl{f}\leq \beta(\wl{g})$  and $h^{-1}f^{n}h = g$ for some $h\in G$ and $1\leq n\leq N$.
It is finite, non-empty (since $g\in\mathcal F$), and can be computed (since conjugacy is decidable).
Let $f$ be an element of $\mathcal{F}$ such that the exponent $n$ is maximum amongst elements of $\mathcal{F}$ and find an element $h\in G$ such
that $h^{-1} f^{n} h = g$ (we may find $h$ by enumeration).

We claim that if $\overline{g}$ is a generator of $C(g)$ then either $h^{-1}f h=\overline{g}$, or $h^{-1}f h = \overline{g}^{-1}$.  Indeed,
$h^{-1} f h\in C(g)$ since it commutes with $g=(h^{-1} f h)^{n}$, hence, $h^{-1} f h=\overline{g}^k$ for some $k$ and so
\begin{equation*}
g = (h^{-1}f h)^{n}=\overline{g}^{kn}.
\end{equation*}

Suppose $k> 0$.  Since $\overline{g}^{kn}=g$,  (\ref{Osin1}) implies that $\overline{g}$ is conjugate to some element $g_{0}$ with $\wl{g_{0}}\leq \beta(\wl{g})$.
Then $g_{0}^{kn}$ is conjugate to $g$, so by (\ref{Osin2}) $kn\leq N$, hence, $g_{0}\in \mathcal{F}$.
By maximality of the exponent in the choice of $f$, $k$
must be 1 and $h^{-1} f h=\overline{g}$.  If $k< 0$, a similar argument shows that $h^{-1}f h = \overline{g}^{-1}$.
\end{proof}

\begin{definition}\label{rel-q}
Let G be a group generated by a finite set $X$, $\{P_1,\ldots ,P_m\}$ be
a collection of subgroups of $G$. A subgroup $R$ of $G$ is called relatively quasi-
convex with respect to $\{P_1,\ldots ,P_m\}$ (or simply relatively quasi-convex when
the collection $\{P_1,\ldots ,P_m\}$ is fixed) if there exists a constant $\sigma >0$ such that
the following condition holds. Let $f, g$ be two elements of $R$, and $p$ an arbitrary
geodesic path from $f$ to $g$ in $Cayley (G,X\cup \mathcal P)$, where $\mathcal P$ is the union of all subgroups  in $\{P_1,\ldots ,P_m\}$. Then for any vertex $v\in p$, there
exists a vertex $w\in R$ such that
$$dist _X(u,w)\leq\sigma .$$
Note that, without loss of generality, we may assume one of the elements $f, g$ to
be equal to the identity since both the metrics $dist _X$ and $dist _{X\cup\mathcalP}$ are invariant
under the left action of $G$ on itself.\end{definition}
It is easy to see that, in general, this definition depends on $X$. However in
case of relatively hyperbolic groups it does not depend on $X$ \cite{osin}.

\begin{prop}\label{int}\cite{KMqc}
Let $G$ be a  toral relatively hyperbolic  group  given  a finite presentation, and $H$ and $R$
be  finitely
generated  relatively quasi-convex subgroups   of  $G$  (given by finite generating sets).

1) There exists an algorithm which solves the membership problem for $H$;

2) There is a finite family ${\mathcal
J}$ of non-trivial intersections $J=H^g\cap R\neq 1$ such that any
non-trivial intersection $H^{g_1}\cap R$ has form $J^r$ for some
$r\in R$ and $J\in {\mathcal J}.$ If $H$ and $R$ have  peripherally finite index or trivial intersection with conjugates of peripheral subgroups, then  one can effectively find 
generators of the subgroups from ${\mathcal J}.$

\end{prop}

\begin{cor}
\label{cy:conjugate-into} Let $H,R$ be finitely generated  relatively quasi-convex
subgroups of a toral relatively hyperbolic group $G$. If $H$ and $R$ have  peripherally finite index or trivial intersection with conjugates of peripheral subgroups, 
then  one can effectively verify whether or not $R$ is conjugate
into $H$, and if it is, then find a conjugator.
\end{cor}
\begin{proof} Using Proposition \ref{int} we can find all non-trivial intersections $J=R\cap g^{-1}Hg$.  For each such  subgroup $J$ we check if all  generators of $R$ belong to $J$. 
If such $J$ exists, then $J=R$ and $R$ is conjugate into $H$. If such $J$ does not exist, then $R$ is not conjugate into $H$. 
\end{proof}

\subsection{JSJ decomposition of toral relatively hyperbolic groups}

\begin{definition} \label{primary} A splitting of a group $G$ is a graph of groups decomposition. The
splitting is called abelian if all of the edge groups are abelian.
An elementary splitting is a graph of groups decomposition for which the underlying
graph contains one edge. A splitting is reduced if it admits no edges carrying an amalgamation
of the form $A*_C C$. 

Let $G$ be a toral relatively hyperbolic group. A reduced splitting
of $G$ is called essential if

(1) all edge groups are abelian; and

(2) if $E$ is an edge group and
$x^k\in E$ for some $k>0$ then
$x\in E$. 

A reduced splitting
of $G$ is called primary if it is essential and all noncyclic abelian subgroups of $G$ are elliptic
(that is, conjugate into vertex subgroups of the splitting).\end{definition}
\begin{prop} \cite{DG} There is an algorithm which takes a finite presentation for a 
toral relatively hyperbolic group as input, and outputs its Grushko decomposition. There is also an algorithm which takes a finite presentation for a freely
indecomposable toral relatively hyperbolic group, $\Gamma$ say, as input and outputs a graph
of groups which is a primary JSJ decomposition for $\Gamma$.\end{prop}

\begin{prop} \label{propDG} (\cite{DG},Theorem 3.35)  Suppose that $G$ is a toral relatively hyperbolic group, and that $\Lambda$ is a primary splitting of $G$. Then, every vertex group of $\Lambda$ is toral relatively hyperbolic, and its parabolic subgroups are the intersections of the parabolic subgroups of $G$ with the vertex group.
\end{prop}


\begin{prop} \label{propD} (\cite{bw}, Lemma 4.9 or Remark 2.4) In the proposition above the vertex groups embed
as  relatively quasi-convex subgroups in $G$, therefore they embed quasi-isometrically.  
\end{prop}
Indeed,  edge groups in a primary splitting of $G$ are direct summands in maximal abelian subgroups, therefore they are relatively quasi-convex in $G$. By \cite{bw}, Lemma 4.9, vertex groups are also relatively quasi-convex. We can also apply \cite{bw}, Remark 2.4, where the family $\mathbb Q$ is the family of maximal abelian non-cyclic subgroups of $G$. 
By \cite{osin}, Theorem 4.13 they are quasi-isometrically embedded.

\begin{definition} We say that a representation  of a group as the fundamental group of a graph of groups is {\em relative} to a collection of subgroups $H_1,\ldots ,H_n$ or {\em modulo} a collection of subgroups if all subgroups in the collection are conjugate into vertex groups. In this case we say that a splitting is reduced modulo $H_1,\ldots ,H_n$ if it admits no edges carrying an amalgamation of the form  $A*_C C,$  such that $C$ does not contain a conjugate of any of the subgroups $H_1,\ldots ,H_n$. \end{definition}

\begin{prop} \label{prop6} There is an algorithm which takes a finite presentation for a 
 toral relatively hyperbolic group $G$ and finitely generated  subgroups $H_1,\ldots ,H_n$ of   $G$ as input, and outputs a graph
of groups which is a primary abelian JSJ decomposition for $G$ relative to $H_1,\ldots ,H_n$.\end{prop}
\begin{proof} We consider the case of $n=1$ because the general case can be proved similarly. We first construct a  primary JSJ decomposition $D$ of $G$ relative to the set of generators of $H$ (each generator is elliptic). This can be done using \cite{DG}.
We will need the following lemma.

\begin{lemma} Given a primary elementary abelian splitting of a freely indecomposable toral relatively hyperbolic group $G$ and its finitely generated subgroup
$H$, there is an algorithm to decide if $H$ is elliptic in this splitting and if it is not, to produce an element from $H$ which is hyperbolic.
\end{lemma}
\begin{proof} 
Suppose $G=A*_CB$ is an elementary abelian splitting. By the CSA property of toral relatively hyperbolic groups (see \cite{Gr}, Lemma 2.5), and \cite{GKM}, $C$ must be maximal abelian either in $A$ or in $B$. Each generator of $H$ can  be written in an amalgamated product normal form
as $h=a_1b_1\ldots a_kb_k$ with $a_i\in A, b_i\in B.$ Notice that every element that is conjugate to a cyclically reduced element in $G$ can be obtained from this element by a cyclic permutation post composed with conjugation by an element from $C$ (see, for example, \cite{LS}). Therefore if an element is conjugate into $A$, then the reduced form of some of its cyclic permutations belongs to $A$.
There is an algorithm to decide this because we only have to be able to solve the membership problem in $C$ for $A$ and $B$, and since $A$ and $B$ are toral relatively hyperbolic, and $C$ is relatively quasi-convex (as a direct summand of a maximal abelian subgroup), the membership problem in $C$ is decidable \cite{KMqc}.
 If none of the generators is conjugate into $A$ or $B$, then the lemma is proved.  Otherwise,   we can conjugate $H$ and suppose that the first (nontrivial) generator of $H$ belongs, say, to $A$.  If the first generator belongs to $A$, then $H$ is conjugate into $A$ if and only if all the other generators belong to $A$. 
 
HNN extensions can be considered similarly, this proves the lemma.
\end{proof}

By induction we can prove the following.
\begin{cor} Given a primary  abelian splitting of a freely indecomposable toral relatively hyperbolic group $G$ and its finitely generated subgroup
$H$, there is an algorithm to decide if $H$ is elliptic in this splitting and if it is not, to produce an element from $H$ which is hyperbolic.
\end{cor}

We can now finish the proof of the proposition.  Suppose $G$ is a non-trivial free product, and $H$ is not elliptic with respect to the free (Grushko) decomposition of $G$,  $G=G_1*\ldots *G_t*F_r$.  We can construct a hyperbolic cyclically reduced element $h\in H$. We construct a relative free decomposition  adding the maximal cyclic subgroup containing $h$ to the set of peripheral subgroups $G=\bar G_1*\ldots *\bar G_k*F_s$.  The number of  terms in the free decomposition decreases. Indeed, let $T$ be a Bass-Serre tree corresponding to the second decomposition. Then every $G_i$ is conjugate into some $\bar G_j$ (otherwise $G_i$ would be freely decomposable).  Moreover, if the normal form of $h$ contains syllables from $G_i$ and $G_k$, then $G_i$ and $G_k$ must be conjugate into the same subgroup $\bar G_j$. 
Therefore $k<t$. 

 We can repeat  checking if $H$ is elliptic in  the obtained free decomposition until we obtain a free decomposition of $G$ such that $H$ is elliptic.  

Now we can consider a freely indecomposable factor containing $H$. For simplicity we will assume  that $G$ is freely indecomposable. Let $D$ be a JSJ decomposition of $G$. If $H$ is not elliptic with respect to an elementary splitting corresponding to some edge connecting a rigid subgroup and a QH vertex group $Q$ (with the corresponding surface $S_Q$) of $D$, then  we can construct a hyperbolic element  $h\in H$.

We need to recall the following definition.
A  pair $(g, |\chi |)$ of  genus and  absolute value of the Euler characteristic of the surface corresponding to a QH subgroup $Q$ is called the {\em size} of $Q$ (and is denoted $size(Q)$).  A tuple $$size(D)= (size(Q_1),\ldots,size(Q_n))$$  of sizes of the MQH subgroups of the decomposition $D$ of a freely indecomposable group in decreasing order (defined in Section 2.3, \cite{KMel}),   is called the {em regular size} of this decomposition.  We compare sizes left lexicographically.

Following \cite{DG} we construct a relative primary JSJ decomposition $D_h$ adding the maximal cyclic subgroup containing $h$ to the set of peripheral subgroups. All elementary splittings of $G$ corresponding to $D_h$ are also elementary splittings corresponding to $D$, but since $\langle h\rangle $ is elliptic in $D_h$ and hyperbolic in $D$, there are strictly less
elementary  splittings corresponding to $D_h$ than to $D$.  Moreover, not all elementary splittings corresponding to simple closed curves on $Q$ and edges from $Q$ are splittings of $D_h$. Indeed, application of some of the canonical Dehn twists of $S_Q$ (corresponding to essential simple closed curves on $S_Q$) would change $h$. Since generators of the Mapping class group of $S_Q$ are Dehn twists along  particular non-separating simple  closed curves on $S_Q$, (see \cite{L}, \cite{H}), some of   these simple closed curves either do not belong to sub-surfaces of $S_Q$ corresponding to QH subgroups of $D_h$, or are not essential.  Therefore these are proper sub-surfaces, and the size(s) of  the QH subgroup(s) in $D_h$ that replace QH-subgroup $Q$ (if any) is strictly less than $size(Q)$. Then  $size (D_h)<size (D).$ 

We now check if $H$ is elliptic in the quadratic decomposition corresponding to $D_h$ (quadratic decomposition is obtained from $D_h$ by collapsing all the edges between non-QH subgroups). If $H$ is not elliptic, we find a hyperbolic element $h_1$ and construct a decomposition $D_{h,h_1}$ adding the maximal cyclic subgroup containing $h_1$ to the set of peripheral subgroups.
Similarly we have $size(D_{h,h_1})< size(D_h)$.
 Since the regular size cannot decrease infinitely, eventually we will obtain a splitting $D_Q$ such that MQH subgroups of this splitting are exactly MQH subgroups of the JSJ decomposition relative to $H$.  

We can now check if $H$ is elliptic in the decomposition obtained from $D_Q$ by collapsing all the edges between non-abelian subgroups. We then increase if necessary the edge groups between rigid and abelian subgroups. This operation  decreases the abelian rank $ab(D)$ (the sum of ranks of the abelian vertex groups minus the sum of ranks of the edge groups between abelian and rigid subgroups). Again, we cannot decrease the abelian rank infinitely, so eventually we get a decomposition $D_{Q,A}$ of $G$ such that $H$ is elliptic with respect to all QH and abelian subgroups of this decomposition.

We now  collapse those edges between rigid subgroups of $D_{Q,A}$ that correspond to elementary splittings of $G$  for which $H$ is not elliptic.  The obtained decomposition will be automatically reduced  (in relative sense), essential, and, since maximal abelian subgroups are elliptic, primary.   
\end{proof}

\section{Effective description of homomorphisms to $\Gamma$}

In this section, following \cite{KMac} and \cite{KMel} we describe an algorithm that takes as input a finite system
of equations $S$ over a torsion-free hyperbolic group $\Gamma$ and produces a tree diagram $\mathcal{T}$ that encodes the set
$\Hom_{\Gamma}(\Gamma_{R(S)},\Gamma)$.  When $S$ is a system without coefficients, we interpret $S$ as relators for a finitely presented group
$G=\langle X|S\rangle$ and the diagram $\mathcal{T}$ encodes instead the set $\Hom(G,\Gamma)$.

There are two ingredients in this construction: first,
the reduction of the system $S$ over $\Gamma$ to finitely many systems of equations over free groups, and second, the construction of Hom-diagrams (Makanin-Razborov diagrams)
for systems of equations over free groups.

Fix $\Gamma=\GammaPresentation$  a finitely presented torsion-free hyperbolic group, $F$ the free group on $A$,
and $\pi:F\rightarrow \Gamma$ the
canonical epimorphism.

Denote $F[X]=F(X)\ast F(A),\ \Gamma [X]=F(X)\ast \Gamma .$ The map $\pi$ induces an epimorphism $F[X]\rightarrow \Gamma[X]$, also denoted $\pi$, by fixing each $x\in X$.  For a system of equations
$S\subset F[X]$, we study the corresponding system $S^{\pi}\subset\Gamma[X]$ which we  denote again by $S$.
The radical of $S$ over $\Gamma$ (the normal subgroup that consists of all elements of $\Gamma [X]$ that are sent to the identity by all solutions of $S$) will be denoted $R{(S)}$.  

The coordinate group is defined as $\Gamma_{R(S)}=\Gamma[X] / R(S)$, where $X$ is precisely the set of variables appearing
in $S$.


Let $\overline{\phantom{c}}$ denote the canonical epimorphism $F(X,A)\rightarrow \Gamma_{R(S)}$.
For a homomorphism $\phi: F(X,A)\rightarrow K$ we define $\overline{\phi}: \Gamma_{R(S)} \rightarrow K$
by
\[
\overline{\phi}\big(\overline{w}\big) =  \phi(w),
\]
where any preimage $w$ of $\overline{w}$ may be used.  We will always ensure that $\overline{\phi}$ is a well-defined homomorphism.


\subsection{Reduction to generalized equations over free groups}
In \cite{RS95}, the problem of deciding whether or not a system of equations $S$ over a torsion-free hyperbolic group $\Gamma$ has a solution was solved by constructing
\emph{canonical representatives} for certain elements of $\Gamma$. This construction reduced the problem to deciding the existence of solutions in finitely many
systems of equations over free groups, which had been previously solved.  The reduction may also be used to find all solutions to $S$ over $\Gamma$, as described
below. We will use exponential notation for composition of homomorphisms. The notion of a {\em generalized equation} can be found in \cite{Imp}.

\begin{lemma}\label{Lem:RipsSela1}
Let $\Gamma=\GammaPresentation$ be a torsion-free $\delta$-hyperbolic group and $\pi : F(A)\rightarrow \Gamma$ the canonical epimorphism.  There
is an algorithm that, given a system $S(Z,A)=1$
of equations over $\Gamma$, produces finitely many generalized equations

\begin{equation}\label{eq:1}
S_{1} (X_{1},A),\ldots,S_{n}(X_{n},A)
\end{equation}
over $F$ (each $S_{1} (X_{1},A)$ can be also considered as  a system of equations $S_{1} (X_{1},A)=1$ over $F$),
constants $\lambda, \mu > 0$, and homomorphisms $\rho_{i}: F(Z,A)\rightarrow F_{R(S_{i})}$ for $i=1,\ldots,n$
such that
\begin{romanenumerate}
\item for every $F$-homomorphism $\phi : F_{R(S_{i})}\rightarrow F$,  the map $\overline{\rho_{i}\phi\pi}:\Gamma_{R(S)}\rightarrow \Gamma$ is a $\Gamma$-homomorphism, and
\item for every $\Gamma$-homomorphism $\psi: \Gamma_{R(S)}\rightarrow \Gamma$ there is an integer $i$ and an $F$-homomorphism
$\phi : F_{R(S_{i})}\rightarrow F(A)$ such that $\overline{\rho_{i}\phi\pi}=\psi$. Moreover, for any $z\in Z$, the word $z^{\rho _i\phi} $ labels a $(\lambda,\mu)$-quasigeodesic path for $z^{\psi} $ in $\Gamma$, and $\phi$ is a solution of 
the generalized equation $F_{R(S_{i})}$.
\end{romanenumerate}
Further, if $S(Z)=1$ is a system without coefficients, the above holds with $G=\GPresentation$ in place of $\Gamma_{R(S)}$ and `homomorphism' in place of
`$\Gamma$-homomorphism'.
\end{lemma}

\begin{proof}
The result is an easy corollary of Theorem~4.5 of \cite{RS95}, but we will provide a few details.

We may assume that the system $S(Z,A)$, in variables $z_{1},\ldots,z_{l}$, consists of $r$ constant equations and $q-r$ triangular equations, i.e.
\[
S(Z,A) = \braced{z_{\sigma(j,1)}z_{\sigma(j,2)}z_{\sigma(j,3)}=1}{j=1,\ldots,q-r}{z_{s}  =  a_{s}}{s=l-r+1,\ldots,l}
\]
where $\sigma(j,k)\in\{1,\ldots,l\}$ and $a_{i}\in\Gamma$.  A construction is described in \cite{RS95}
which, for every $m\in\naturals$, assigns to each element $g\in \Gamma$ a word $\theta_{m}(g)\in F$ satisfying
\[
\theta_{m}(g)=g \mbox{ in } \Gamma
\]
called its \emph{canonical representative}.  The representatives $\theta_{m}(g)$ are not `global canonical representatives', but do satisfy
useful properties for certain $m$ and certain finite subsets of $\Gamma$, as follows.

Let\footnote{The constant of hyperbolicity $\delta$ may be computed from a presentation of $\Gamma$ using the results of \cite{Papa}.} $L=q\cdot 2^{5050(\delta+1)^{6}(2|A|)^{2\delta}}$.
Suppose $\psi: F(Z,A)\rightarrow \Gamma$ is a solution of $S(Z,A)$ and denote
\[
\psi(z_{\sigma(j,k)})=g_{\sigma(j,k)}.
\]
Then there exist
$h_{k}^{(j)}, c_{k}^{(j)}\in F(A)$ (for $j=1,\ldots,q-r$ and $k=1,2,3$) such that
\begin{romanenumerate}
\item each $c_{k}^{(j)}$ has length less than\footnote{The bound of $L$ here, and below, is extremely loose.  Somewhat tighter, and more intuitive, bounds are given in \cite{RS95}.} $L$
 (as a word in $F(A)$), \label{RepsCond1}
\item $c_{1}^{(j)}c_{2}^{(j)}c_{3}^{(j)}  =  1$ in $\Gamma$, \label{RepsCond2}
\item there exists $m\leq L$\ such that the canonical representatives satisfy the following equations in $F$:\label{RepsCond3}
\begin{eqnarray}
\theta_{m} (g_{\sigma(j,1)}) & = & h_{1}^{(j)} c_{1}^{(j)} \left(h_{2}^{(j)}\right)^{-1} \label{CanonReps1}\\
\theta_{m} (g_{\sigma(j,2)}) & = & h_{2}^{(j)} c_{2}^{(j)} \left(h_{3}^{(j)}\right)^{-1}\\
\theta_{m} (g_{\sigma(j,3)}) & = & h_{3}^{(j)} c_{3}^{(j)} \left(h_{1}^{(j)}\right)^{-1}.\label{CanonReps3}
\end{eqnarray}
\end{romanenumerate}
In particular, when $\sigma(j,k)=\sigma(j',k')$ (which corresponds to two occurrences in $S$ of the variable $z_{\sigma(j,k)}$) we have
\begin{equation}
h_{k}^{(j)} c_{k}^{(j)} \left(h_{k+1}^{(j)}\right)^{-1} = h_{k'}^{(j')} c_{k'}^{(j')} \left(h_{k'+1}^{(j')}\right)^{-1}.\label{Hequality}
\end{equation}

Moreover, $\theta_{m} (g_{\sigma(j,i)})$ labels a $(\lambda,\mu)$-quasigeodesic. Consequently, we construct the systems $S(X_{i},A)$ as follows.
For every $m$ and every choice of $3(q-r)$ elements $c_{1}^{(j)},c_{2}^{(j)},c_{3}^{(j)}\in F$  ($j=1,\ldots,q-r$)
satisfying (i) and (ii)\footnote{The word problem in hyperbolic groups is decidable.}
we build a system $S(X_{i},A)$
consisting of the equations
\begin{eqnarray}
x_k^{(j)}c_{k}^{(j)}\left(x_{k+1}^{(j)}\right)^{-1} & = & x_{k'}^{(j')}c_{k'}^{(j')}\left(x_{k'+1}^{(j')}\right)^{-1} \label{Eqn:SC1}\\
x_k^{(j)}c_{k}^{(j)}\left(x_{k+1}^{(j)}\right)^{-1} & = & \theta_{m}(a_s) \label{Eqn:SC2}
\end{eqnarray}
where an equation of type   (\ref{Eqn:SC1}) is included whenever  $\sigma(j,k)=\sigma(j',k')$ and an equation of type (\ref{Eqn:SC2}) is included whenever
$\sigma(j,k)=s\in\{l-r+1,\ldots,l\}$.  To define $\rho_{i}$, set
\[
\rho_i (z_s) = \bracedTwo{x_k^{(j)}c_{k}^{(j)}\left(x_{k+1}^{(j)}\right)^{-1},}{1\leq s \leq l-r \mbox{ and } s=\sigma(j,k)}{ \theta_{m}(a_s),}{l-r+1\leq s \leq l}
\]
where for $1\leq s \leq l-r$ any $j,k$ with $\sigma(j,k)=s$ may be used.

If $\psi: \Gamma _{R(S)}\rightarrow \Gamma$ is any solution to $S(Z,A)=1$, there is a system $S(X_{i},A)$ such that $\theta_{m}(g_{\sigma(j,k)})$ satisfy
(\ref{CanonReps1})-(\ref{CanonReps3}).  Then the required solution $\phi$ is given by
\[
\phi\big(x_{j}^{(k)}\big) = h_{j}^{(k)}.
\]
Indeed, (iii) implies that $\phi$ is a solution to $S(X_{i},A)=1$.  For $s=\sigma(j,k)\in\{1,\ldots,l-r\}$,
\[
z_{s}^{\rho_{i}\phi} = h_{k}^{(j)} c_{k}^{(j)} \left(h_{k+1}^{(j)}\right)^{-1} = \theta_{m}(g_{\sigma(j,k)})
\]
and similarly for $s\in\{l-r+1,\ldots,l\}$, hence $\psi= \rho_{i}\phi\pi$.

Conversely, for any solution $\phi\big(x_{j}^{(k)}\big)= h_{j}^{(k)}$ of $S(X_{i}, A)=1$ one sees that by (\ref{Eqn:SC1}),
\[
z_{\sigma(j,1)}z_{\sigma(j,2)}z_{\sigma(j,3)} \rightarrow^{\rho_{i}\phi} h_{1}^{(j)} c_{1}^{(j)}c_{2}^{(j)}c_{3}^{(j)} \big(h_{1}^{(j)}\big)^{-1}
\]
which maps to 1 under $\pi$ by (ii), hence $\rho_{i}\phi\pi$ induces a homomorphism.
\end{proof}


\subsection{Encoding solutions with the trees $\mathcal{T}$ and ${\mathcal T}(S,\Gamma )$}\label{section:HomDiagrams}
In this section we describe how to encode solutions to equations over $\Gamma$ using a $Hom$-diagram. We begin by describing a $Hom$-diagram for generalized equations \cite{Imp} over free groups. There is an algorithm described in \cite{Imp} that, given a generalized equation $\Omega(X,A)$ over the free group $F=F(A)$, constructs 
a diagram, which encodes the solutions of $\Omega$. Let $G$ be the coordinate group $F_{R(\Omega)}$ of $\Omega$ considered as a system of equations over $F$. Specifically, the algorithm constructs a directed, finite, rooted tree $T$ that has the following properties:
\begin{romanenumerate}
\item Each vertex $v$ of $T$ is labelled by a pair $(G_v,Q_v)$, where $G_v$ is an $F$-quotient of $G$ and $Q_v$ the subgroup of canonical automorphisms in $Aut_F(G_v)$ corresponding to a splitting of $G_v$ as a fundamental group of a graph of groups, that we find from the Elimination process of $\Omega$. The root $v_o$ is labelled by $(F_{R(\Omega)},1)$ and every leaf is labelled by $(F(Y)*F(A),1)$ 
where $Y$ is some finite set (called \emph{free variables}). Each $G_v$, except possibly $G_{v_0}$, is fully residually $F$.
\item Every (directed) edge $v\rightarrow v'$ is labelled by a proper surjective\\$F$-homomorphism $\pi(v,v'):G_v\rightarrow G_{v'}$.
\item For every $\phi\in Hom_F(F_{R(\Omega)},F)$, that is a solution of $\Omega$ (that must be non-cancellable in $F$) there is a path $p=v_0v_1\ldots v_k$, where $v_k$ is a leaf labelled by $(F(Y)*F(A),1)$, elements $\sigma_i\in Q_{v_i}$, and a $F$-homomorphism $\phi_0:F(Y)*F(A)\rightarrow F(A)$ such that
\begin{equation}\label{*}\phi=\pi(v_0,v_1)\sigma_1\pi(v_1,v_2)\sigma_2\ldots\pi(v_{k-2},v_{k-1})\sigma_{k-1}\pi(v_{k-1},v_k)\phi_0\end{equation}
Considering all such $F$-homomorphisms $\phi_o$,  the family of the above sequences of homomorphisms   is called the \emph{fundamental sequence} over $F$ corresponding to $p$.

Considering all such $F$-homomorphisms $\phi_o$ that produce solutions of $\Omega$ the family of the solutions of $\Omega$ factoring through the above fundamental sequence is called the \emph{fundamental sequence for  the generalized equation} $\Omega$ over $F$ corresponding to $p$.
\item  The  splitting of each fully residually free group $G_v$ is its Grushko decomposition followed by the abelian splittings of the factors that are found by the elimination process.  If $C_{vi}$ is such a factor, then the splitting is not necessarily the JSJ decomposition of $C_{vi}$, but it is maximal in a sense that it encodes all elementary abelian splittings of $C_{vi}$ that can be found by the elimination process, and has maximal QH and abelian vertex groups that can be found by the elimination process for $G_v$.

\end{romanenumerate}

Notice that not all the homomorphisms that factor through (\ref{*}) are solutions of $\Omega$, but they all are homomorphisms from $F_{R(\Omega)}$ to $F$.

Let $S(Z,A)=1$ be a finite system of equations over $\Gamma$.  We will construct a diagram $\mathcal{T}$ to encode the set of solutions of $S(Z,A)=1$. Namely, we  will construct a tree of  fundamental sequences encoding all solutions of a system $S(Z,A)=1$ of equations over $\Gamma$ using the tree of fundamental sequences for `covering' systems of generalized equations over $F$ constructed in Lemma~\ref{Lem:RipsSela1}.

We apply Lemma~\ref{Lem:RipsSela1} and construct the generalized equations $S_{1}(X_{1},A),\ldots,S_{n}(X_{n},A)$.
We create a root vertex $v_{0}$ labelled by $F(Z,A)$.  For each of the generalized equations $S_{i}(X_{i},A)$, let $T_{i}$
be the tree constructed above. Build an edge from $v_{0}$ to the root of $T_{i}$ labelled by the homomorphism $\rho_{i}$.  For each leaf $v$ of $T_{i}$, labelled by $F(Y)\ast F$, build a new vertex $w$ labelled by $F(Y)\ast\Gamma$ and an edge $v\rightarrow w$ labelled by the homomorphism
$\pi_{Y}:F(Y)\ast F\rightarrow F(Y)\ast \Gamma$ which is induced from $\pi:F\rightarrow \Gamma$ by acting as the identity on $F(Y)$.

Define a \emph{branch} $b$ of $\mathcal{T}$ to be a path $b=v_{0} v_{1} \ldots v_{k}$ from the root $v_{0}$ to a leaf $v_{k}$.
Let $v_{1}$ be labelled by $F_{R(S_{i})}$ and $v_{k}$ by $F(Y)\ast \Gamma$.
We associate to $b$ the set $\Phi_{b}$ consisting of all homomorphisms $F(Z, A)\rightarrow \Gamma$ of the form
\begin{equation}\label{9}
\rho_{i}\pi(v_{1},v_{2}) \sigma_{2} \cdots \pi(v_{k-2},v_{k-1})\sigma_{k-1}\pi(v_{k-1},v_{k})\pi_{Y}\psi ,
\end{equation}
where $$\rho_{i}\pi(v_{1},v_{2}) \sigma_{2} \cdots \pi(v_{k-2},v_{k-1})\sigma_{k-1}\pi(v_{k-1},v_{k})$$ is a solution of the generalized equation and
 $\sigma_{j}\in Q_{v_{j}}$ and $\psi\in \mathrm{Hom}_{\Gamma}(F(Y)\ast\Gamma,\Gamma)$.
Since $\mathrm{Hom}_{\Gamma}(F(Y)\ast\Gamma,\Gamma)$ is in bijective correspondence with the set of functions $\Gamma^{Y}$, all elements of $\Phi_{b}$
can be effectively constructed.  We have obtained the following result.

\begin{prop}\cite{KMac} \label{Thm:EffectiveSolutions}
There is an algorithm that, given a system $S(Z,A)=1$ of equations over $\Gamma$, produces a diagram encoding its set of solutions.  Specifically,
\[
\Hom_{\Gamma}(\Gamma_{R(S)},\Gamma) = \{ {\phi} \sst \phi\in\Phi_{b}\cs \mbox{$b$ is a branch of $\mathcal{T}$}\}
\]
where $\mathcal{T}$ is the diagram described above.  When the system is coefficient-free, then the diagram encodes $\Hom(G, \Gamma)$ where
$G=\GPresentation$.
\end{prop}

Note that in the diagram $\mathcal{T}$, the groups $G_{v}$ appearing at vertices are not quotients of the coordinate group $\Gamma_{R(S)}$ and
that to obtain a homomorphism from $\Gamma_{R(S)}$ to $\Gamma$ one must compose maps along a complete path ending at a leaf of $\mathcal{T}$.  In \cite{Gro05} it
is shown that for any toral relatively hyperbolic group there exist Hom-diagrams with the property that every group $G_{v}$ is a quotient of
$\Gamma_{R(S)}$ and that every edge map $\pi(v,v')$ is a
proper surjective homomorphism.

\begin{definition}
A {\em fundamental sequence} or a {\em fundamental set} of homomorphisms over $\Gamma$ corresponding to the diagram $$\Gamma_{R(S)}\rightarrow _{\pi_0} G_1\rightarrow _{\pi_1}  G_2\rightarrow\ldots \rightarrow _{\pi_{n-1}}  G_n=F\ast \Gamma \ast H_1\ast\ldots \ast H_k$$

where
\begin{enumerate}
\item $H_1,\ldots , H_k$ are freely indecomposable groups isomorphic to subgroups of $\Gamma$,  and $G_1,\ldots ,G_n$ are $\Gamma$-limit groups.
 
\item $\pi_i, 0<i<n-1$ are fixed proper epimorphisms,  $\pi _{n-1}$ is an epimorphism but  may not be proper.
 
\item The homomorphisms in this sequence are compositions $\pi_0\sigma_1\pi _1\ldots\sigma _{n-1}\pi _{n-1}\tau$, where $\sigma _i$ is a canonical automorphism of $G_i$ corresponding to a Grushko decomposition of $G_i$ followed by some abelian decompositions of the freely indecomposable factors where  all non-cyclic abelian subgroups are elliptic.
Canonical automorphisms are identical on the free factor of this Grushko decomposition.
 
\item  $\tau$ is a homomorphism that maps each $H_i$ monomorphically  into a conjugate of a fixed subgroup of $\Gamma$ (and for each $H_i$ it is a fixed monomorphism followed by a conjugation) and maps $F$ into $\Gamma$.
\end{enumerate}
\end{definition}\begin{definition}
\begin{enumerate} A fundamental sequence defined above is called {\em strict}  if it has the following properties:
\item The image of each non-abelian vertex group of $G_i$ under $\pi _i$ is non-abelian.
\item For each $1\leq i< n$, $\pi_i$ is injective on rigid subgroups, edge groups, and subgroups generated by the images of edge groups in abelian vertex groups in $G_{i-1}$.
\item For each $1\leq i <n $, if $R$ is a rigid subgroup in $G_{i}$ and $\{A_j\}$, $1\leq j\leq m$, the abelian vertex groups in $G_{i}$ connected to $R$ by edge groups $E_j$ with the maps $\eta_j:E_j\rightarrow A_j$, then $\pi_i$ is injective on the subgroup $\langle R,\eta_1(E_1),\ldots,\eta_m(E_m)\rangle$ which we will call the {\em envelop of $R$}.
\item The images of different factors in the Grushko decomposition of $G_i$ under $\pi _i$  are different factors in the free decomposition of $G_{i+1}$. 
\end{enumerate}
\end{definition}

\begin{prop} \label{prop7} \label{prop:7} There is an algorithm to replace each family of homomorphisms $\Phi _b$ constructed in Proposition \ref{Thm:EffectiveSolutions} with an NTQ group corresponding to a strict fundamental sequence over $\Gamma$.  \end{prop}

\begin{remark} 
In Sela's terminology \cite{Sela1}, fundamental sequences are called {\em resolutions} and fundamental sequences with these properties that correspond to canonical Makanin-Razborov diagram are called {\em strict resolutions}.  In \cite{Sela}, in the proof of Theorem 7.1 he either uses this term differently or makes some erroneous statements. In \cite{KMel} we used  fundamental sequences with the properties 1-2, but we should have  had also property 3.  We will borrow the term `strict'  for fundamental sequences with the  properties 1-4.
\end{remark}

\begin{proof}  The proposition  was proved in  \cite{KMac}, Section 3.3  (embeddable NTQ system in \cite{KMac}  is a system satisfying  the conditions of the proposition).
We will briefly recall the proof of this proposition here.

The terminal group of each fundamental sequence 
\begin{equation}
\rho_{i}\pi(v_{1},v_{2}) \sigma_{2} \cdots \pi(v_{k-2},v_{k-1})\sigma_{k-1}\pi(v_{k-1},v_{k})\end{equation}
is $G_v=F(Y)*F$.  Applying the homomorphism $\pi_{Y}:F(Y)*F\rightarrow F(Y)*\Gamma$, we
 replace it with $F(Y)*\Gamma$. Then we change the next from the bottom level of the fundamental sequence, the group $G_{v_{k-1}}$ and replace the epimorphism $\pi (v_{k-1},v_k)$ by   $\pi (v_{k-1},v_k)\pi _Y.$
 If, after we replaced $G_v$ with $F(Y)*\Gamma$, there are no collapses of the abelian JSJ decomposition $D_{k-1}$ of $G_{v_{k-1}}$  (situations when one of the conditions in Proposition \ref{prop:7} breaks),  then we construct  two-level canonical NTQ group with the bottom level $F(Y)*\Gamma$, and with all QH and abelian vertex groups corresponding to the second level from the bottom  the same as for the group $G_{v_{k-1}}$. Canonical automorphisms corresponding to these QH and abelian vertex groups will be the same as they are for  $G_{v_{k-1}}$,  Dehn twists  associated to the edges of the decomposition of the new group will correspond to the Dehn twists associated to the edges of $D_{k-1}$.  Notice, that we just changed one vertex group replacing $F(A)$ by $\Gamma$ but this does not change other vertex groups,

 We now consider all possible collapses of the abelian JSJ decomposition $D_{k-1}$ of $G_{v_{k-1}}:$
\begin{enumerate}
\item an edge group in $D_{k-1}$ is mapped by $\pi (v_{k-1},v_k)\phi$ to the trivial group,  in particular, some of the boundary elements of a MQH subgroup are mapped to a trivial group,
\item some MQH subgroup is mapped to an abelian group,

\item some non-abelian vertex group is mapped to an abelian group.
\end{enumerate}

Suppose there are collapses.   We first collapse all QH subgroups that are mapped into the identity and adjacent edges. We replace all QH vertex groups that are mapped to non-trivial abelian subgroups with the centralizers of their images. 
Suppose $Q$ is an MQH subgroup that is mapped to a non-abelian group. We replace by the identity all the boundary elements $p_i$ that are mapped to the identity.  We remove then from the graph all the edges that are connected to this MQH vertex group and such that edge group is mapped into the identity. This is equivalent to filling with disks the corresponding boundary components of the surface $S_Q$, and remove the corresponding edge. The group of automorphisms associated with this MQH subgroup will also be changed, it will be generated by Dehn twists corresponding to simple closed curves on the new surface.  If the obtained subgroup is not a QH subgroup anymore we consider it as a rigid subgroup (in case it is given by the relations $x^2y^2=1, x^2c_1^zc_2=1$ or $c_1^{z_1}c_2^{z_2}c_3^{z_3}=1$) or as an HNN extension of a rigid subgroup (in case $x^2y^2c^z=1$).  If we  remove an edge corresponding to a stable letter, then we add the cyclic subgroup generated by a stable letter as a free factor.

For each edge  between any two non-QH vertex groups for which the edge group is mapped to the trivial element we collapse this edge so that two vertices become one and replace the obtained new vertex group by its image in $F(Y)*\Gamma$.  We collapse  those abelian  subgroups that are mapped to the trivial element. We collapse the edges between those non-QH non-abelian subgroups that are mapped to the abelian subgroup and are connected by these edges to other rigid subgroups. 

 For each connected component of the obtained graph of groups that does not contain $\Gamma$ we add a new letter so that the fundamental group of that connected component is mapped into the  conjugate of $F(Y)*\Gamma$ by this new letter.  At the bottom we have a free product of a free group, group $\Gamma$ and some conjugates of $\Gamma$. With the obtained decomposition we associate the $\Gamma$-NTQ group in a standard way (by taking the iterated free extension of centralizers of the images of the edge groups in the bottom group (and consecutive groups) by new letters, and joining together some abelian vertex groups to make the group commutative transitive.  This NTQ group is toral relatively hyperbolic by \cite{D}.

We can continue this way and change consecutively all the higher levels of the fundamental sequence. On each level for each connected component of the obtained graph of groups that does not contain $\Gamma$ we add a new generator so that the fundamental group of this connected component is mapped into $\Gamma*F(Y)$ conjugated by the new generator. This completes the construction of fundamental sequences and corresponding NTQ groups over $\Gamma$.

The construction is algorithmic, because all what is needed is to solve the word problem in the group on the lower level, that is, by induction,  toral relatively hyperbolic. 
\end{proof}

We denote by ${\mathcal T} (S, \Gamma)$ the constructed tree of strict fundamental sequences over $\Gamma$.  We just proved the following:

\begin{prop}\label{KMac} Each  branch of ${\mathcal T} (S, \Gamma)$ has the following form $$\Gamma_{R(S)}\rightarrow _{\psi_0} L_1\rightarrow _{\psi_1}  L_2\rightarrow\ldots \rightarrow _{\psi_{n-1}}  L_n=F\ast \Gamma \ast K_1\ast\ldots \ast K_k,$$  where

1) $K_1,\ldots , K_k$ are freely indecomposable groups isomorphic to freely indecomposable factors   in the Grushko decomposition   of $\Gamma$, 

2) $L_1,\ldots ,L_n$ are NTQ $\Gamma$-limit groups, $\psi _0$ is a fixed homomorphism; $\psi_i$ for $0<i\leq n-1$ is a fixed proper epimorphism that is  retraction on $L_{i+1}$, 

3) There is a strict fundamental sequence assigned to each branch. The homomorphisms in this sequence are compositions $\psi _0\sigma_1\psi _1\ldots\sigma _{n-1}\psi _{n-1}\tau ,$ where $\sigma _i$ is a canonical automorphism of $L_i$ with respect to the cyclic splitting where non-QH non-abelian subgroups are factors in the Grushko decomposition of $L_{i+1}$, and $\tau$ is a homomorphism that maps each $K_i$ monomorphically onto a conjugate of the corresponding subgroup of $\Gamma$ (and for each $K_i$ it is a fixed monomorphism followed by a conjugation), and maps $F$ into $\Gamma$.

Every homomorphism from $\Gamma _{R(S)}$ to $\Gamma$ factors through one of the fundamental sequences corresponding to the branches of ${\mathcal T} (S, \Gamma)$.  Factors in the Grushko decomposition of $L_i$ are mapped into different factors in the Grushko decomposition of $L_{i+1}.$ The construction is algorithmic.
\end{prop}

Similarly one can prove the following result, where groups with no sufficient splitting modulo subgroups were defined in \cite{KMel}.

\begin{prop} \label{prop8} Suppose we are given a finitely presented group $G=\langle Z|S\rangle $ and a finite number of finitely generated subgroups $H_1,\ldots ,H_k$
of $G$. There is an algorithm  to construct a finite number of $\Gamma$-NTQ groups corresponding to strict completed fundamental sequences over $\Gamma$ ending with one of the  finitely generated $\Gamma$-limit groups $\{K_i, i\in I\}$ with no sufficient splitting modulo the images of $H_1,\ldots ,H_k$,  such that every homomorphism from $G$ to $\Gamma$ factors through  one  of these groups $\{K_i, i\in I\}$ followed by a homomorphism from the corresponding $K_i$ to $\Gamma$. Groups $K_i$ are given by generating sets as subgroups of  $\Gamma$-NTQ groups. 
\end{prop}
\begin{proof} We can consider the relations of $G$ as a system of equations $S(Z)=1$ over $\Gamma$, construct corresponding  generalized equations $S_i(X_i)$ over a free group using canonical representatives.  The groups $F_{R(S_i)}$ will then contain lifts of the subgroups $H_1,\ldots ,H_k$ generated by some pre-images of their generators. We can construct fundamental sequences for each  $F_{R(S_i)}$ relative lifts of the subgroups $H_1,\ldots ,H_k$ and apply the reworking process as in the proof of the previous proposition, to transform these fundamental sequences into fundamental sequences over $\Gamma$ relative to the subgroups $H_1,\ldots ,H_k$.

\end{proof}
\begin{remark} In Sela's terminology  rigid and solid limit groups are similar objects to groups with no sufficient splitting.
\end{remark}

\begin{prop} \label{relh} If $G$ is a toral relatively hyperbolic group, and $S(X,A) = 1$ a system of equations having a solution in $G$, then there exists an algorithm to construct a finite number $G$-NTQ groups  that encode all solutions of $ S(X,A) = 1$ in $G$.
\end{prop}
\begin{proof} The proof is similar to the proof of Proposition \ref{prop7}. We assign to each solution of $S(X,A)=1$ canonical representatives in the free product $\tilde G$ of a free group and parabolic subgroups  (using \cite{D1}, Theorem 3.1) and a disjunction of systems of equations over $\tilde G$. By \cite{CRK}, the solution set of a system of equations in $\tilde G$ is described by a finite number of fundamental sequences over $\tilde G$. We will have a statement similar to Proposition \ref{Thm:EffectiveSolutions}. Then we replace each fundamental sequence with a fundamental sequence over $G$ using the reworking process as we did in the proof of Proposition \ref{prop7}, where the free group $F$ should be replaced by the free product $\tilde G$ and the hyperbolic group $\Gamma$ should be replaced by the group $G$.
\end{proof}
 We define a partial order on the set of $\Gamma$-limit groups $\{H_{b_i}=\phi_{b_i}(\Gamma_{R(S)})\leq N_{b_i}\}$ over all branches $b_i$ of ${\mathcal T}(S,\Gamma )$, with $N_{b_i}$ the $\Gamma$-NTQ group corresponding to $b$. It's often notationally convenient to denote $H_{b_i},N_{b_i},\phi_{b_i}$ by $H_i,N_i,\phi_i$.  For given elements $H_1,H_2$ we say that $H_2\leq _{Sol} H_1$  if   for every  homomorphism  $\psi_2: H_2\rightarrow \Gamma$  that factors through $N_2$  there exists  a homomorphism  $\psi_1: H_1\rightarrow \Gamma$ that factors through $N_1$ such that $\phi _1\psi _1= \phi _2\psi _2 $.  In this case the  canonical map $\phi _2: \Gamma _{R(S)}\rightarrow H_2$ can be split  as $\tau\phi _1,$  where $\phi _1$ is the canonical homomorphism $\phi _1: \Gamma _{R(S)}\rightarrow H_1$ and $\tau$ is a $\Gamma$-epimorphism from $H_1$ to $H_2$. 
 
\begin{remark}\label{!!!} (\cite{KMT}, Proposition 6)  There is an algorithm to find all homomorphisms $\phi _b: \Gamma _{R(S)}\rightarrow N_b$,  where $N_b$ is a $\Gamma$-NTQ group corresponding to a strict fundamental sequence in ${\mathcal T} (S, \Gamma)$ for a Sol-maximal $\Gamma$-limit quotient $H$ of $\Gamma _{R(S)}$ such that $H=\phi _b(\Gamma _{R(S)})$.\end{remark}

\subsection{Quasi-convex closure of subgroups of $\Gamma$}
Let $\Gamma$ be a torsion free non-elementary hyperbolic group, $H$ a subgroup of $\Gamma$ given by generators. We will describe a certain procedure that either finds a splitting of $H$, or  constructs a group $K$, $H\leq K\leq \Gamma$ such that $K$ is quasi-convex in $\Gamma$ and rigid relative to $H$ (every cyclic splitting of $K$ induces a cyclic splitting of $H$).  We will call this $K$ a {\em quasi-convex closure of $H$}. 

If $H$ is not elliptic in any  cyclic splitting of $\Gamma$, then $\Gamma$ is a quasi-convex closure of $H$.  If $H$ is elliptic in a cyclic splitting of $\Gamma$, we consider instead of $\Gamma$ the vertex group $\Gamma_1$ (in the cyclic JSJ decomposition of $\Gamma$ modulo $H$) containing a conjugate of $H$. Such a subgroup $\Gamma_1$ is quasi-convex and hyperbolic as a vertex group  by Proposition \ref{propD}, and the quasi-convexity constants can be found effectively. Then we consider a splitting of $\Gamma_1$, and continue this process until $H$ is not elliptic in any cyclic decomposition of the corresponding vertex group $\Gamma_j$. Then the  conjugate $K$ of $\Gamma_j$ that contains $H$ is  the   quasi-convex closure of $H$.
Hierarchical accessibility for hyperbolic groups is proved in \cite{louder-touikan}.  Notice that $K$  is quasi-convex, therefore hyperbolic.

\subsection{A complete set of canonical NTQ groups}\label{complete}

We thank D. Groves and H. Wilton for finding some inaccuracies in the previous version of this section \cite{GW1}. 

We will use the following result without further references.
\begin{prop} (\cite{Sela}, Theorem 1.10) \label{10}  Let $\Gamma$ be a non-elementary torsion-free hyperbolic
group, and let $L$ be a non-abelian, freely indecomposable, (strict)
$\Gamma$-limit group (that is not a subgroup of $\Gamma$, see \cite{Sela}, Definition 1.2). Then $L$ admits an essential cyclic splitting.
If, furthermore, $L$ contains  $\Gamma$ as the coefficient subgroup,  then the splitting may
be chosen so that the coefficient subgroup is elliptic.\end{prop}

In the previous section we constructed the tree of strict fundamental sequences (or a $Hom$-diagram) encoding all solutions of a finite system $S(Z,A)=1$ of equations over $\Gamma$ using the tree of fundamental sequences for ``covering'' systems of equations over $F$.

In this section we describe a {\em canonical $Hom$-diagram} that encodes all the homomorphisms from $\Gamma _{R(S)}$ into $\Gamma$. This diagram is  a tree of canonical strict fundamental sequences (in \cite{KMel}, Section 7.6, a similar tree was called  the (augmented) canonical embedding tree $T_{CE}(F_{R(S)})$), and corresponding NTQ systems.   





 The $Hom$-diagram is a tree.  For each  branch of this tree  \begin{equation}\label{canhom}\Gamma_{R(S)}\rightarrow _{\pi_0} G_1\rightarrow _{\pi_1}  G_2\rightarrow\ldots \rightarrow _{\pi_{n-1}}  G_n=F\ast \Gamma \ast H_1\ast\ldots \ast H_k,\end{equation}   there is a strict fundamental sequence assigned. Here 
 
 1) $H_1,\ldots , H_k$ are freely indecomposable groups isomorphic to subgroups of $\Gamma$, 
 
 2) $\pi_i, 0<i<n-1$ are fixed proper epimorphisms,  $\pi _{n-1}$ may not be proper.
 
 3) The homomorphisms in this sequence are compositions $\pi_0\sigma_1\pi _1\ldots\sigma _{n-1}\pi _{n-1}\tau ,$ where $\sigma _i$ is a canonical automorphism of $G_i$ corresponding to a Grushko decomposition of $G_i$ followed by the JSJ decompositions of the freely indecomposable factors, 
 
 4)  $\tau$ is a homomorphism that maps each $H_i$ monomorphically  onto a conjugate of a fixed subgroup of $\Gamma$ (and for each $H_i$ it is a fixed monomorphism followed by a conjugation) and maps $F$ into $\Gamma$.

The existence of such a diagram  can be obtained from \cite{Sela} and \cite{KMac}.
 Indeed, the difference between the diagram that we described and the diagram constructed in \cite{Sela} is that in \cite{Sela} the homomorphism $\pi _{n-1}$ must be proper and  homomorphism  $\tau$ can be an arbitrary embedding of  groups $H_1,\ldots , H_k$ into $\Gamma$. In our diagram, if $H_i$ has a non-trivial abelian splitting,  then $H_i$ appears as a factor in the Grushko decomposition of $G_{n-1}$, and in the leaves of the diagram we have a fixed (up to conjugacy) monomorphism of $H_i$ into a conjugate of $\Gamma$.  Notice that (non-canonical) fundamental sequences from  Proposition \ref{KMac}   terminate with fixed conjugacy classes of monomorphisms of $K_t$ into   $\Gamma$,  and we are going to use this fact to prove that  canonical fundamental sequences  satisfy the same property. Notice also that by \cite{GW}, Lemma 7.2, $H_i$ has infinitely many conjugacy classes of monomorphisms into $\Gamma$ if and only if $H_i$ has a non-trivial cyclic splitting.  
 We  define two  embeddings of $H_i$ into $\Gamma$ to be equivalent if one is obtained from the other by pre-composing with an automorphism of $H_i$ generated by Dehn twists corresponding to cyclic splittings of $H_i$ and post-composing with a conjugation. 
 
We claim that there is a finite number of such equivalence classes.
Suppose to the contrary, that there is an infinite family of non-equivalent  monomorphisms.  Let $\{h_i\}$ be a fixed system of generators of $H$.  Consider a sequence $(\alpha _j:H\rightarrow \Gamma )$ of non-equivalent minimal monomorphisms (corresponding to 
minimal in the equivalence class sum of lengths of the images of $\{h_i\}$).  After passing to a subsequence, the sequence can be taken to be stable, it converges to a stable isometric action of $H$ on a real tree. Since $H=H/Ker_{\rightarrow}(\alpha _j),$  one can use Rips machine to show that $H$ has an essential cyclic splitting such that some of the monomorphisms in the sequence can be shorten 
by pre-composing with Dehn's twists corresponding to this splitting. Since the sequence consists of minimal monomorphisms,
we have a contradiction with the assumption about infinite number of equivalence classes.
 
  We can assume (combining foldings and slidings) that all JSJ decompositions have the property that each vertex with non-cyclic abelian vertex group that is connected to a rigid subgroup is connected to only one rigid subgroup.
 
We assign an NTQ system to this branch as follows.   First, replace each subgroup $H_i$ that is not a hyperbolic closed surface group by its quasi-convex closure $\Gamma _i$, then $G_n$ is replaced by  $\hat G_n=F\ast \Gamma \ast \Gamma _1\ast\ldots \ast\Gamma _k,$  $G_n\leq\hat G_n$. Notice that $\hat G_n$ is a hyperbolic group. Let $D_{n-1}$ be an abelian JSJ decomposition  of $G_{n-1}$ (we mean the decomposition into a free product of  freely indecomposable factors followed by the JSJ decompositions of the factors).  We order the edges $e_1,\ldots e_k$ between   rigid subgroups.  We extend the centralizer of the image of  the edge group of $e_1$ in $\hat G_n$ by a new letter,  and obtain a new group ${\hat G_n}^{(1)}$, then iteratively for each $i\leq k$ we extend by a new letter the centralizer of the image of the edge group of $e_i$ in  ${\hat G_n}^{(i-1)}$ (see \cite{KMel} for precise description). 

Then the fundamental group of the subgraph of groups generated by the rigid subgroups in $D_{n-1}$ will be embedded into this iterated centralizer extension $\tilde G_{n}={\hat G_n}^{(k)}$ of $\hat G_{n}$. We also attach abelian vertex groups of $D_{n-1}$ to $\tilde G_n$ the following way.  Consider edges (with  maximal cyclic edge groups) that connect  non-cyclic abelian vertex groups  in $D_{n-1}$  to a non-abelian non-QH vertex group. Some of  the centralizers of the images of the edge groups may become conjugate in $\tilde G_{n}$.  We join all edges with conjugate centralizers of the images of the edge groups into an equivalence class. For each equivalence class we do the following. Denote the edges in the class by $\bar e_1,\ldots , \bar e_s$. Let $m$ be the sum of the ranks of abelian vertex groups connected to $\bar e_1,\ldots , \bar e_s$. We extend only the centralizer of the image of the edge group corresponding to $\bar e_1$ by new $m-s$ commuting letters (free rank $m-s$ extension of a centralizer defined in the introduction).   Denote the obtained group by $\dddot G_n$. We attach QH subgroups of $D_{n-1}$  identifying boundary components  with their images in $\bar G_n$, and obtain the new group $\overline G_{n-1}$ such that $G_{n-1}$ is embedded into   $\overline G_{n-1}$.   Notice that  since $G_{n-1}$ is a $\Gamma$-limit group,  edge groups corresponding to edges adjacent to  QH subgroups are maximal cyclic in QH subgroups. 

The group $\dddot G_n$ is a $\Gamma$-limit group as the iterated extension of centralizers  of the $\Gamma$-limit group  $\hat G_n$.  The group $\overline G_{n-1}$ is a $\Gamma$-limit group because it is a fundamental group of a family of regular quadratic equations over $\dddot G_n$.  The group $\overline G_{n-1}$ is an NTQ group by definition. It is also toral relatively hyperbolic by Dahmani's Combination Theorem 0.1, items (1) and (2) \cite{D}. 

We denote $\hat G_n$ by $N_n$ and $\overline G_{n-1}$ by $N_{n-1}$, then $G_{n-1}\leq N_{n-1}.$  

Suppose we have already constructed the group $N_{i}$. We will show now how to construct $N_{i-1}$.   Let $D_{i-1}$ be an abelian JSJ decomposition  of $G_{i-1}$ (we mean the decomposition into a free product of  freely indecomposable factors followed by the JSJ decompositions of the factors).  For each freely indecomposable factor of $G_{i-1}$  that is not a closed surface group and not a free abelian group  we perform the construction in parallel and then take the free product of the constructed groups and all the factors that are  closed surface groups and  free abelian groups. 

To simplify notation we now suppose that $G_{i-1}$ is freely indecomposable and $\bar D_{i-1}$ is its abelian JSJ decomposition.  Let $\pi_{i-1}:G_{i-1}\rightarrow G_i$ be the canonical  homomorphism from diagram (\ref{canhom}), 
and when we are talking about images of elements  $g\in G_{i-1}$, we mean $\pi_{i-1}(g)$ that belongs to $G_i$, $N_i$, and any group containing $G_i$ as a subgroup.   

We order the edges $e_1,\ldots e_k$ between   rigid subgroups in $\bar D_{i-1}$.  We freely extend the centralizer of the image of  the edge group of $e_1$ in $N_i$ by a new letter,  and obtain a new group ${N_i}^{(1)}$, then iteratively for each $j\leq k$ we freely extend by a new letter the centralizer of the image of the edge group of $e_i$ in the previously constructed  group ${N_i}^{(j-1)}$.

Then the fundamental group of the subgraph of groups generated by the envelopes of rigid subgroups in $\bar D_{i-1}$ will be embedded into this iterated centralizer extension $\tilde N_{i}={N_i}^{(k)}$ of $N_i$. We also attach abelian vertex groups of $\bar D_{i-1}$ to $\tilde N_i$ the following way.  Consider edges that connect  non-cyclic abelian vertex groups  in $\bar D_{n-1}$  to  non-abelian non-QH vertex groups. Some centralizers of the images of the edge groups may become conjugate to centralizers of  some other edge groups in $\tilde N_i$.  We put two edges into the same equivalence class if some conjugates of the images of their edge groups in $\tilde N_i$ commute.  For each equivalence class we do the following. Denote the edges in the class by $\bar e_1,\ldots , \bar e_s$. Let $m$ be the sum of the ranks of abelian vertex groups connected to $\bar e_1,\ldots , \bar e_s$, and $p$ the sum of the ranks of their direct summands containing finite index subgroups generated by edge groups. We extend only the centralizer of the image in $\tilde N_i$ of the edge group corresponding to $\bar e_1$ by new $m-p$ commuting letters.  Denote the obtained group by $\dddot N_i$. We attach QH subgroups of $\bar D_{n-1}$  to $\dddot N_i$ identifying boundary components  with their images in $\tilde N_i$ that is a subgroup of $\dddot N_i$, and obtain the new toral relatively hyperbolic NTQ $\Gamma$-limit group $N_{i-1}$ such that $G_{i-1}$ is embedded into   $N_{i-1}$.   

 We construct iteratively the  group  $N=N_1$, which is NTQ, and, therefore, a $\Gamma$-limit group,  toral relatively hyperbolic  and contains $G_1$ as a subgroup.

 The set of all NTQ groups corresponding to a canonical $Hom$-diagram is {\em a complete set of canonical NTQ groups}. We often also consider a complete set of {\em canonical $\Gamma $-NTQ groups}, when the bottom level  is  a free product  of $F$ and several  conjugates of $\Gamma$ by new generators.  Namely, in the beginning of the construction of the canonical $\Gamma$-NTQ group we can take this free product istead of $F*\Gamma *\Gamma _1*\ldots *\Gamma _k$, and then apply the construction.  
 
\begin{definition}  If we have an NTQ group $N_H$ over the group $H$. Then we can increase the group $H$ by $K$ and construct the NTQ  group $N_K$ over $K$ such that $N_H\leq  N_K$  using the construction above. We call $N_K$ the {\em commutative transitive closure} of $N_H$ obtained by extending $H$ by $K$.
\end{definition}  

The problem of abelian edges becoming conjugated happens unavoidably when there are parameters.
  
\begin{definition} The group of canonical automorphisms of the  NTQ group $N_i$  $(i=1,\ldots,n-1)$ is the group of canonical automorphisms  with respect to the Grushko decomposition of $N_i$ followed by an abelian splitting  such that MQH subgroups correspond to MQH subgroups of $N_i$  (and $G_i$), abelian vertex groups correspond to the abelian vertex groups of $N_i$  and non-QH non-abelian vertex groups are freely indecomposable factors in the Grushko decomposition of the group $N_{i+1}$. 
\end{definition} 

For each fundamental sequence (\ref{canhom}) of the canonical Hom-diagram we assign the following fundamental sequence of the {\em completed} canonical Hom-diagram
 \begin{equation}\label{canhom1}\Gamma_{R(S)}\rightarrow _{\pi_0} N_1\rightarrow _{\bar\pi_1}  N_2\rightarrow\ldots \rightarrow _{\bar\pi_{n-1}}  N_n=F\ast \Gamma \ast \Gamma _1\ast\ldots \ast \Gamma_k,\end{equation}

 1)  $\bar\pi_i, 0<i\leq n-1$ is a retraction on $N_{i+1},$
 
 2) The homomorphisms in this {\em completed} fundamental sequence are compositions $\pi_0\bar\sigma_1\bar\pi _1\ldots\bar\sigma _{n-1}\bar\pi _{n-1}\tau ,$ where $\bar\sigma _i$ is a canonical automorphism of $N_i$,
 
 3)  $\tau$ is a homomorphism that maps each $\Gamma _i$ monomorphically  onto a conjugate of a fixed subgroup of $\Gamma$ (and for each $\Gamma _i$ it is a fixed monomorphism followed by a conjugation) and maps $F$ into $\Gamma$.
 
 All $\Gamma$-homomorphisms from $\Gamma _{R(S)}$ to $\Gamma$ that factor through the fundamental sequence (\ref{canhom}) naturally factor through 
 (\ref{canhom1}).

\subsection{Correction to Sela's theorem about formal solutions \cite{Sela}}\label{ImFT}

We will formulate a theorem that is similar to  the Parametrization Theorem, also called the Implicit function theorem (\cite{Imp},Theorem 12) for free groups. A similar result is also formulated  in \cite{Sela}, Theorem 2.3 for hyperbolic groups, but the formulation in \cite{Sela} contains an error as one can see from the  counter-example at the end of this section.

 Let $S(X,A)=1$ be an NTQ  system over a non-elementary torsion free hyperbolic group $\Gamma$.  Let $N$ be the corresponding  NTQ group.  Suppose a formula  $\Psi=T(X,Y,A) = 1 \ \wedge \ W(X,Y,A) \neq 1$ is compatible 
with $S(X,Y)=1$. {\em Corrective extensions} of $N$ are obtained by consecutively performing the following steps:

(i) Replacing each of the free abelian groups that appear in the Grushko  decompositions on different levels of the NTQ group by a free abelian group of the same rank (depending on $\Psi$), that contains the original one as a subgroup of finite index.

Replacing each of the free abelian vertex groups that appear in the abelian decompositions of freely indecomposable factors on different levels of the NTQ group by a free abelian group of the same rank (depending on $\Psi$), that contains the original one as a subgroup of finite index.

(ii) replacing each of the terminal factors $H_i$ by a  freely indecomposable quasi-convex  subgroup $\Gamma _i$ of $\Gamma$, where $H_i\leq\Gamma_i$ and the fixed monomorphism from $H_i$ to the conjugate of $\Gamma$ is extended to the fixed  monomorphism from $\Gamma_i$ to this conjugate of $\Gamma$.  Subgroups $\Gamma_i$ depend on $\Psi$.

(iii)  Embedding the obtained group into the commutative transitive group using the  procedure similar to  the construction of the  NTQ group for a  strict fundamental sequence in Section \ref{complete}. We do have a completed strict fundamental sequence corresponding to the NTQ system $S(X,A)=1$ here, but since the subgroups at the base level were extended, we have to complete it again because   centralizers of some abelian vertex groups that were not conjugate,  may become conjugate. 

We define the abelian size of $N$, denoted $ab(N)$, as the sum of the ranks of the abelian vertex groups in decompositions corresponding to different levels of $N$ minus the sum of the ranks of their direct summands containing edge groups as subgroups of finite index. Then $ab(N)$ is the same as $ab(N_{corr})$ for each corrective extension $(N_{corr})$ of $N$.

\begin{theorem}\label{IFT} (Implicit Function Theorem) Let $S(X,A)=1$ be an NTQ  system over a non-elementary torsion free hyperbolic group $\Gamma$. Suppose a formula  $T(X,Y,A) = 1 \ \wedge \ W(X,Y,A) \neq 1$ is compatible 
with $S(X,Y)=1$. Then this formula admits a lift into  a family of $NTQ$ groups $N_1,\ldots ,N_k$ that are corrective extensions of $\Gamma _{R(S)}$, and  toral relatively hyperbolic. Every solution from the  fundamental sequence of solutions of $S(X,Y)=1$ (see Remark 1) factors through  the  fundamental sequence for the NTQ system corresponding to   one of the corrective extensions.
\end{theorem}

{\bf A counter-example to \cite{Sela}, Theorem 2.3.}

We will be considering $\Gamma$-limit groups for a torsion free hyperbolic group $\Gamma$.

Let $H=\langle h_1,\ldots ,h_k|R(h_1,\ldots ,h_k)=1\rangle$, where $R=1$ is a finite set of relations,  be a rigid hyperbolic group with trivial outer automorphism group.  Let $u=u(h_1,\ldots ,h_k), v=v(h_1,\ldots ,h_k)$ be non-conjugate elements from $H$ that are not proper powers.
Let $\Gamma$ be isomorphic to $\langle H,c|u^c=v\rangle$, then $\Gamma$ is hyperbolic by the combination theorem. Then the group $L=\Gamma\ast L_1$, where $L_1=\langle H,y_1,y_2|[u,y_1]=[v,y_2]=1\rangle $ is a restricted $\Gamma$-limit group, this presentation is the   JSJ decomposition of $L_1$, and $\Gamma\ast H$ is the shortening quotient.  So the strict resolution of $L$  is $L\rightarrow \Gamma\ast H$, and the completion of this resolution according to 
\cite{Sela2}, Definition 1.11 (there is a reference to this definition right before Definition 2.2 in \cite{Sela}) is the same resolution, where $L$ is its completed limit group. We now consider the sentence $$\forall h_1,\ldots ,h_k, y_1, y_2\exists x ((R(h_1,\ldots ,h_k)=1\wedge [u,y_1]=[v,y_2]=1)$$ $$
\implies (u^x=v\vee W(h_1,\ldots,h_k)=1)),$$
Where $W(h_1,\ldots,h_k)= 1$ is the disjunction of a finite set of equations  (homomorphisms from $L_1$ to $\Gamma$ that are not injective on $H$ satisfy one of the finite set of equations). Notice that a disjunction of two equations in a torsion free hyperbolic group is equivalent to a system of equations. This sentence is true in $\Gamma$ because there is a unique up to conjugacy embedding of $H$ into $\Gamma$, and we can take for $z$ the corresponding conjugate of $c$.  But such $z$ does not exists in any extension of $L$ obtained from $L$ 
as it is said in \cite{Sela} by
``(i) replacing each of the (free) abelian vertex groups that appear in the various abelian decompositions associated with the completion ($L$) by (free) abelian supergroups that contain ones as subgroups of finite index) and (ii) Replacing each of the factors $H_j$ by a freely indecomposable group $V_j$ with an associated embedding $H_j\rightarrow V_j$, and $V_j$ is isomorphic to a subgroup of $\Gamma$." 

Indeed, to force such $u$ and $v$ from $H$ to become conjugates we have to extend $H$ by $\Gamma _1$, that is an isomorphic copy of $\Gamma$, but in this case
the group $K=\langle \Gamma, y_1, y_2| [u,y_1]=[v,y_2]=1\rangle$  even if we replace it by $K_1$ where we replace the abelian vertex groups by two supergroups containing them as subgroups of finite index, becomes not commutation transitive, because $u=v^c.$
Therefore $K$ or $K_1$ is not a $\Gamma$-limit group, and $\Gamma\ast K\rightarrow \Gamma\ast\Gamma _1$ is not a (well-structured) resolution,
not a resolution at all.

This contradicts to the statement of \cite{Sela}, Theorem 2.3. Notice that there is no proof of Theorem 2.3 in \cite{Sela}.

\subsection{Algorithmic results from \cite{KMT}}
\begin{prop}  (\cite{KMT}, Theorem 2) \label{11}Let $S(Z, A)=1$ be a finite system of equations over $\Gamma$. There is an algorithm to construct  a complete set of corrective extensions of canonical $\Gamma$-NTQ groups that are, in particular,  toral relatively hyperbolic $\Gamma$-limit groups, and associated canonical Hom-diagrams (\ref{canhom1}) and (\ref{canhom})  for $\Gamma _{R(S)}$.  

1)  NTQ-groups $N_i$ in  the completed diagram with branches  (\ref{canhom1}) are constructed by their finite presentations, and canonical groups of automorphisms are given by their generators, and we know their presentation. 

2)  $\Gamma$-limit groups $G_i$ in the diagram with branches (\ref{canhom}) are constructed by defining their generators inside corresponding $NTQ$-subgroups $N_i$, and canonical automorphisms of these $\Gamma$-limit groups are  induced on their  freely indecomposable factors by canonical automorphisms of the corresponding groups $N_i$.\end{prop}

 The proposition is proved using generalized equations for canonical representatives.  

\subsection{Generic Families}\label{subsection:gen-fam}
To detect splittings of $\Gamma$-limit groups we will need  the notion of a generic family of solutions of an NTQ system. It is given in \cite{KMel} (Definition 22) and is very technical. To make this paper more self contained we will give a definition here, we also give it in the language of \cite{Sela2}, Definition 1.5,  because we will have to refer to some results from \cite{Sela}. Our generic family contains a test sequence from \cite{Sela2}.

Let $S(X) = 1$ be a system of equations with a solution in a group $G$. We say that a system of equations $T(X,Y)
= 1$ is {\em compatible} with $S(X) = 1$ over $G$, if for every
solution $U$ of $S(X) = 1$ in $G$, the equation $T(U,Y) = 1$ also
has a solution in $G$. More generally, a formula $\Phi(X,Y)$ in
the language $L_A$ is {\em compatible} with $S(X) = 1$ over $G$,
if for every solution $\bar{a}$ of $S(X) = 1$ in $G$ there exists
a tuple $\bar{b}$ over $G$ such that the formula $\Phi(\bar a,
\bar b )$ is true in $G$.

Suppose now that a formula $\Phi(X,Y)$ is compatible with $S(X)=
1 $ over $G$. We say that $\Phi(X,Y)$ {\em admits a lift to a
generic point} of $S(X) = 1$ over $G$ (or just that it has an {\em $S$-lift} over
$G$), if the formula $\exists Y \Phi(X,Y)$ is true in
$G_{R(S)}$ (here $Y$ are variables and $X$ are constants from
$G_{R(S)}$). Finally, an equation $T(X,Y) = 1$, which is
compatible with $S(X) = 1$, admits a {\em complete $S$-lift} if
every formula $T(X,Y) = 1 \ \& \ W(X,Y) \neq 1$, which is
compatible with $S(X) = 1$ over $G$, admits an $S$-lift. We say
that the lift (complete lift) is {\em effective} if there is an
algorithm to decide for any equation $T(X,Y)=1$ (any formula
$T(X,Y) = 1 \ \& \ W(X,Y) \neq 1$) whether $T(X,Y)=1$ (the formula
$T(X,Y) = 1 \ \& \ W(X,Y) \neq 1$) admits an $S$-lift, and if it
does, to construct a tuple in $G_{R(S)}$ verifying this formula (a solution of $T(X,Y) = 1 \ \& \ W(X,Y) \neq 1$ in $G_{R(S)}$).

We now describe the construction of particular families of solutions, called \emph{generic families}, of a NTQ system which imply nice lifting properties for that system.  This description can be skipped at the first reading.

Consider a fundamental sequence with the corresponding NTQ system $S(X,A)=1$ of depth $N$. We construct generic families iteratively for each level $k$ of the system, starting at $k=N$ and decreasing $k$. There is an abelian decomposition of $G_k$ corresponding to the NTQ structure. Let $V_1^{(k)},\ldots,V_{M_k}^{(k)}$ be the vertex groups of this decomposition given some arbitrary order. We construct a generic family for level $k$, denoted $\Psi(k)$, by constructing generic families for each vertex group in order. We denote a generic family for the vertex group $V_i^{(k)}$ by $\Psi(V_i^{(k)})$. If there are no vertex groups, in other words the equation $S_k=1$ is empty ($G_k=G_{k+1}*F(X_k)$) we take $\Psi(k)$ to be a sequence of growing different Merzljakov's words (defined in \cite{Imp}, Section 4.4).

\begin{remark}\label{Merz}When using generic families in this paper,  by \cite{Sela}, Proposition 2.1,  instead of a family of growing Merzliakov's words  in $\Gamma$, one can just take $\mu_i(H_i)$ conjugated by a new letter, as well as one can take new letters  for the basis of $F(X_k)$.  So instead of a family of homomorphisms in $G$ or $\Gamma$, we can consider generic family as a family of solutions into $G\ast F$ or $\Gamma\ast F$. \end{remark}

If $V_r^{(k)}$ is an abelian group then it corresponds to equations of the form $[x_i,x_j]=1$ or $[x_i,u]=1$, $1\leq i,j\leq s$,  where $u\in U$ runs through generators of a centralizer in $G_{k+1}$. A solution $\sigma$ in $G_{k+1}$ to equations of these forms is called \emph{$B$-large} if there are some $b_1,\ldots,b_s$ with each $b_i> B$ such that $\sigma(x_i)=(\sigma(x_1))^{b_1\ldots b_i}$ or $\sigma(x_i)=u^{b_1\ldots b_i}$, for $1\leq i\leq s$ (possibly renaming $x_1$). A generic family of solutions for an abelian subgroup $V_r^{(k)}$ is a family $\Psi(V_r^{(k)})$ such that for each $B_i$ in any increasing sequence of positive integers $\{B_i\}_{i=1}^{\infty}$ there is a solution in $\Psi(V_r^{(k)})$ which is $B_i$-large.

If  $V_r^{(k)}$ is a QH vertex group of this decomposition, let $S$ be the surface associated to  $V_r^{(k)}$. We associate two collections of non-homotopic, non-boundary parallel, simple closed curves $\{b_1,\ldots b_q\}$ and $\{d_1,\ldots d_t\}$. These collections should have the propery that $S-\{b_1\cup\cdots\cup b_q\}$ is a disjoint union of three-punctured spheres and one-punctured Mobius bands, each of the curves $d_i$ intersects at least one of the curves $b_j$ non-trivially, and their union fills the surface $S$ (meaning that the collection $\{b_1,\ldots b_q, d_1,\ldots ,d_t\}$ has minimal number of  intersections and $S-\{b_1,\cup\cdots\cup b_q\cup d_1\cup\cdots\cup d_t\}$ is a union of topological disks).

Let $\beta _1,\ldots ,\beta_q$ be automorphisms of $V_r^{(k)}$ that correspond to Dehn twists along $b_1,\ldots b_q$ , and $\delta _1,\ldots ,\delta_t$ be automorphisms of $V_r^{(k)}$ that correspond to Dehn twists along $d_1,\ldots d_t$.
We define iteratively a basic sequence of automorphisms $\{\gamma _{L,n}, \phi _{L,n}\}$ (compare with Section 7.1 of \cite{Imp} where one particular basic sequence of automorphisms is used), which is determined by a sequence of  $(t+q)$-tuples $L=\{(p_{1,n},\dots,p_{t,n},m_{1,n},\ldots,m_{q,n})\}_{n=1}^{\infty}$

Let $$\phi _{L,0}=1$$

$$\gamma _{L,n}=\phi_{L,n-1}\delta _1^{m_{1,n}}\ldots \delta _q^{m_{q,n}}, n\geq 1$$

$$\phi _{L,n}=\gamma_{L,n}\beta _1^{p_{1,n}}\ldots \beta _t^{p_{t,n}},n\geq 1$$

Assuming that generic families have already been constructed for $V_i^{(k)}$, $i< r$, and for every vertex group in levels $k'>k$, and that $\Theta _k$ is a family of growing powers of Dehn twists for edges on level $k$, set $\Psi(k')=\Psi(V_{M_{k'}}^{(k')})\Theta _{k'}$ for $k<k'\leq N$ (in other words, the generic family for level $k'$ is the generic family of the last vertex group at that level) and set $\Psi(N+1)=\{1\}$. Let $\pi_k:G_k\rightarrow G_{k+1}$ be the canonical epimorphism. Let $\Sigma_r^{(k)}=\{\psi_{1}\cdots\psi_{r-1}|\psi_{i}\in\Psi(V_i^{(k)})\}$ be the collection of all compositions of generic solutions for previous vertex groups. We then say that $$\Psi(V_r^{(k)})=\{\mu _{L,n,\lambda_n}=\phi _{L,n}\delta _1^{\lambda_n}\ldots \delta _q^{\lambda_n}\sigma_n\pi_k\tau |\sigma_n\in\Sigma_r^{(k)}, \tau\in\Psi(k+1)\}_{n=1}^{\infty}$$ where each $\lambda_n$ is some positive integer, is a generic family for $V_r^{(k)}$ if it has the following property: Given any $n$ and any tuple of positive numbers $\overrightarrow{A}=(A_1,\ldots,A_{nt+nq+1})$ with $A_i< A_j$ for $i< j$, $\Psi$ contains a homomorphism $\mu_{n,L,\lambda_n}$ such that the tuple

\begin{align*}
\overrightarrow{L}_{n,r_n} &=(p_{1,1},\ldots ,p_{t,1}, m_{1,2},\ldots ,m_{q,2},\ldots, m_{1,n},\ldots ,m_{q,n}, p_{1,n},\ldots ,p_{t,n}, \lambda_n) \\
 &= (L_1,\ldots,L_{nt+nq+1})
\end{align*}grows faster than $\overrightarrow{A}$ in the sense that $L_1\geq A_1$ and $L_{i+1}-L_{i}\geq A_{i+1}-A_{i}$.

Finally we set $\Psi(S)=\Psi(V_{M_1}^{(1)})$ to be a generic family of solutions for the $G$-NTQ system $S(X,A)=1$. Notice that $\Psi(S)$ $\Gamma$-discriminates $\Gamma_{R(S)}$. 

One can prove  by inspection of the proof of Theorem 1.18 \cite{Sela2},  and Theorem 2.3 \cite{Sela} the following result.
\begin{prop} If $\Psi(W)$ is a generic family of solutions for a regular
NTQ system $W(X,A)=1$, then for any formula 
$\Phi (X,Y,A) =U(X,Y,A)=1\wedge W(X,Y,A)\neq1$ the following is true: if for any solution $\psi\in\Psi(W)$
there exists a solution of $\Phi(X^{\psi},Y,A)=1$, then $\Phi$ admits a
lift into $\Gamma _{R(W)}.$ If the
NTQ-system $W(X,A)=1$ is not regular, then for any such formula
$\Phi(X,Y,A)=1$ the following is true: if for any solution $\psi\in\Psi(W)$
there exists a solution of $\Phi(X^{\psi},Y,A)=1$, then $\Phi$ admits a
 lift into a family of  corrective extensions of
$\Gamma_{R(W)}$. There is a finite number of these extensions, and any solution of $W(X,A)=1$ factors through one of them.
 \end{prop}

\subsection{Fundamental sequences relative to subgroups}
We first define a {\em quasi-convex closure} of a freely indecomposable $\Gamma$-limit group (inside a corrective extension of a canonical NTQ group). We had such a construction before, but recently a better developed construction of a model $\Gamma$-limit group that suits our needs appeared in \cite{GW1}, therefore we will use it.
 \begin{definition} \label{GW} Let $L$ be a freely indecomposable $\Gamma$-limit group given as a Sol-maximal $\Gamma$-limit quotient of a group $\Gamma _{R(S)},$ where $S=1$ is a finite system of equations over $\Gamma$.
Let $F_{R(\Omega)}$  be the freely indecomposable factor of  the coordinate group  of a generalized equation constructed using canonical representatives 
such that a generic family  $\Psi$ of solutions for a canonical strict fundamental sequence for $L$  over $\Gamma$ is described by this generalized equation.
The fundamental sequence of solutions of $\Gamma_{R(\Omega)}$  corresponding to a generic subfamily  of $\Psi$ can be re-worked as in the proof of Proposition \ref{prop:7} into a completed fundamental sequence of solutions for $\Gamma_{R(\Omega)}$ such that the corresponding NTQ group $N$ over $\Gamma$  is a corrective extension of a canonical  
NTQ  group for $L$ by Proposition \ref{11}. 

If in the re-working process, going from bottom to top, instead of the construction of an NTQ group (as in section \ref{complete}) one applies the  construction of model groups  introduced in \cite{GW1}, then as a result, one algorithmically obtains a toral relatively hyperbolic $\Gamma$-limit group 
$M$ such that $L\leq L_1\leq M\leq N,$  where $L_1$ is a maximal $\Gamma$-limit quotient of $\Gamma _{R(\Omega)}$
that embeds into $N$.

We call this group $M$ a {\em quasi-convex closure} of $L$.

 \end{definition}

\begin{remark}\label{QC} \begin{itemize}
\item [a)] $M$ is not a model group for $L$, it is a model group for a larger group, because to construct a system of equations for canonical representatives in $F$ we add additional variables.
\item [b)] The construction of the complete set of generalized equations does not depend on the way we represent  the system $(Z,A)=1$ over $\Gamma$ as a triangular system. 
\item [c)] Let $M$ be a quasi-convex closure of $L$. Then $M$ is rigid modulo $L$  (\cite{GW1}, Corollary 5.4 implies that $M$ is rigid modulo $L_1$, and $L_1$ is rigid modulo $L$ by construction).
\item [d)] MQH vertex groups of $L$ and $M$ are the same.
\item [e)] The relationship between groups of canonical automorphisms of $M$ and $L$ is the same as described in \cite{GW1}. 
\end{itemize}\end{remark} 
 
 Similarly to Proposition \ref{11} one can prove the following result.

\begin{prop}  \label{12}Let $S(Z, A)=1$ be a finite system of equations over $\Gamma$. There is an algorithm to construct a complete set of corrective extensions of canonical NTQ groups (and systems) for $\Gamma _{R(S)}$ modulo a finite system of finitely generated subgroups $H_1,\ldots ,H_k$ of $\Gamma _{R(S)}.$  

We can define terminal $\Gamma$-limit groups $L_1,\ldots ,L_k$ of these NTQ groups  as maximal limit quotients of coordinate groups of certain systems of equations.
Alternatively, when we need these NTQ groups to be toral relatively hyperbolic, we can replace $L_1,\ldots, L_k$  by their 
quasi-convex closures $M_1,\ldots, M_k$ (see Definition \ref{GW}).  Doing this  replacement we may add an extra level at the bottom of the NTQ group (in case $L_i$ has a splitting but not a sufficient splitting). \end{prop}
\begin{proof} We will give the proof  when the system of subgroups consists of only one subgroup $H$. Using canonical representatives we can construct a family of generalized equations $\Omega _1,\ldots , \Omega _k$ in the free group $F$   ($\pi :F\rightarrow \Gamma$) \cite{Imp} such that each solution of each $\Omega _i$ in $F$ (as a system of equations in the group) corresponds to  a solution of $S(Z,A)=1$ in $\Gamma$, and every solution of $S(Z,A)=1$ in $\Gamma$ corresponds to some solution of some $\Omega _i$ as a generalized equation (solution without cancellations).  We can run the Elimination process for each generalized equation modulo the pre-image of $H$ (generators of $H$ are included in the set of variables of $S(Z,A)=1$).  If a generalized equation corresponds to a generic family of solutions for a freely indecomposable  NTQ system or a freely indecomposable NTQ system modulo a subgroup, then all the splittings on all the levels of this NTQ system are detected in the Elimination process \cite{KMS} (we called it Makanin's process in \cite{Imp}) for the generalized equation, and  produce a corresponding NTQ system over a free group.  Moreover, the edge groups for these splittings are not trivialized in the re-working process described in Section \ref{section:HomDiagrams}. We will obtain in the Elimination process completed  fundamental sequences ending with groups ${K_i^*}$ without sufficient splitting modulo the group $H^*$ generated by $F$ and the variables corresponding to the generators of $H$.   The procedure for finding Sol-maximal $\Gamma$-limit quotients in described in  \cite{KMT}, Section 6.1.
\end{proof}

\section{Decision algorithm for $\forall\exists$-sentences}\label{sec:4}
In this section we will prove Theorem \ref{aetheory}.
 In the rest of the paper we will only consider fundamental sequences satisfying the first and second restrictions from
 \cite{KMel}, Sections 7.8, 7.9. To make this paper self-contained we recall these sections here.

 \subsection{First restriction on fundamental sequences}
\label{spl}  Let $S(Z,A)=1$ be a system of equations over $\Gamma$. We can assume that it is irreducible. We construct fundamental sequences for $S(Z,A)=1$ as in Proposition \ref{11}.  Let
\begin{equation}\label{3.5.} \Gamma _{R(\bar U)}*F(t_1,\ldots
,t_k)=P_1*\ldots *P_q*\langle t_1\rangle *\ldots *\langle t_k\rangle \end{equation} be a reduced
 free decomposition of a maximal shortening quotient $\Gamma _{R(\bar U)}*F(t_1,\ldots
,t_k)$ modulo $\Gamma$ for ${\Gamma}_{R(S)}$ (this shortening quotient is exactly the group corresponding to the second from the top level  of the corresponding fundamental sequence), and $\pi : \Gamma _{R(S)}\rightarrow \Gamma _{R(
\bar U)}\ast F(t_1,\ldots ,t_{\beta _i})$ the  epimorphism.  Let $P$ be the subgroup generated by the
variables (which we denote by $X$) and standard coefficients (that we denote by $C$) of a regular quadratic
equation $Q_{i}=1$ corresponding to some fixed MQH subgroup in the
JSJ decomposition of ${\Gamma}_{R(S)}.$ Consider a free decomposition $\pi
(P)=K_1*\ldots
*K_p*\langle t_{k_{j_1}}\rangle *\ldots *\langle t_{k_{j_2}}\rangle $ inherited from the free
decomposition (\ref {3.5.}) such that each standard coefficient is
conjugated into some $K_j$, and each $K_j$ has a conjugate of some
coefficient. Then there always  is a canonical automorphism that transforms
$X$ into variables $X_1$ with the following properties:

1) the family $X_1$ can be represented as a disjoint union of sets
of variables $X_{11},\ldots, X_{1t}$;

2) every solution of $S=1$ can be transformed by a canonical
automorphism corresponding to $Q_{i}=1$ into a solution of the
system obtained from $S=1$ by replacing $Q_i=1$  by  a system of
several quadratic equations $Q_{i1}(X_{11},C)=1,\ldots ,
Q_{it}(X_{1t},C)=1$ with standard coefficients from $C$;

3)  each quadratic equation $Q_{ij}=1$ either is coefficient-free, or
has coefficients from $C$ which are conjugated into some $K_r$;

4) $X_1^{\pi}$ is a solution of the system
$Q_{i1}(X_{11},C^{\pi_1})=1,\ldots , Q_{it}(X_{1t},C^{\pi_1})=1$;

5) if $Q_{ij}=1$ is coefficient-free, then $X_{ij}^{\pi }$ is a
solution of maximal possible dimension (= rank of the free group in the image) or  the corresponding surface group is a subgroup of $\Gamma$ (unless this surface group is embedded into $\Gamma$ in this fundamental sequence).

6) if  $Q_{ij}=1$ is not coefficient-free, then $Q_{ij}=1$ cannot be transformed by a canonical automorphism corresponding to $Q_{ij}=1$ into an equation 
$$Q_{ij1}(X_{ij1})Q_{ij2}(X_{ij2})=1$$ such that
$Q_{ij1}(X_{ij1})=1$ is coefficient-free and $Q_{ij2}(X_{ij2})=1$
has non-trivial coefficients from $C$ which are conjugated into some
$K_r$ and such that $X_{ij}^{\pi }$ is a solution of the system $Q_{ij1}(X_{ij1})=1$,
$Q_{ij2}(X_{ij1},C^{\pi})=1$.

 Suppose $Q_{ip}=1 $ is some equation in variables $X_{1p}$ in
this family which has coefficients from $C$. Each homomorphism in a
fundamental sequence of homomorphisms from $\Gamma _{R(S)}$ to $\Gamma $ is a
composition $\sigma _1\pi \phi$, where $\sigma _1$ is a canonical
automorphism of $\Gamma _{R(S)}$, and $\phi$ is a homomorphism from
$\Gamma _{R(\bar U)}*F(t_1,\ldots ,t_k)$ to $\Gamma$.

The {\em first restriction} is that we will include into the
fundamental sequence  only   the compositions $\sigma _1\pi \phi$
for which $\{\phi\}$ satisfies the following
property: for each $j$, $Q_{ij}=1$ is not split into a system of two
quadratic equations $Q_{ij1}(X_{ij1})=1$ and $Q_{ij2}(X_{ij2})=1$
with disjoint sets of variables such that $Q_{ij1}(X_{ij1})=1$ is
coefficient-free and $Q_{ij2}(X_{ij2})=1$ has coefficients from $C$ which are
conjugated into some $K_r$ and such that $X_{ij1}^{\pi\phi}$ is a
solution of $Q_{ij1}=1$ and $Q_{ij2}(X_{ij2})^{\pi \phi}=1$.  

This is equivalent to the following construction.  We cut each punctured surface $\Sigma$ corresponding to the MQH subgroup $Q$ in the JSJ decomposition of ${\Gamma}_{R(S)}$ along a  maximal collection of disjoint non-homotopic simple closed curves that corresponds to the lift of the free decomposition $F(t_1,\ldots
,t_k)*P_1*\ldots *P_q$ and are  mapped to the identity on the next level by $\pi$. Moreover, if $\Sigma '$ is a punctured surface obtained from $\Sigma$ and connected to  a rigid subgroup, then when we   adjoin disks to the  boundary components of  $\Sigma '$  that are mapped to the identity, no s.c.c. (simple closed curve) on that surface is mapped to the identity along the fundamental sequence. If $\Sigma '$ 
 is a punctured surface obtained from $\Sigma$ and not connected to  a rigid subgroup, then when we   adjoin disks to the  boundary components of  $\Sigma '$  that are mapped to the identity, we obtain a closed surface, and $\pi _1$ maps it to s free group of maximal dimension or embeds into $\Gamma$.
 
We require a similar property for all levels of the fundamental sequence.

\subsection{Second restriction on fundamental sequences} \label{rk3}
Suppose the family of homomorphisms $\sigma _1\pi _1\ldots \sigma
_n\pi _n\tau$ is a strict completed fundamental sequence, corresponding to the NTQ
system $Q(X_1,\ldots, X_n)=1:$
$$Q_1(X_1,\ldots ,X_n)=1,$$
$$\ldots $$
$$Q_n(X_n)=1$$
adjoint with free variables $t_1,\ldots, t_k$. Here the restriction of $\sigma _i$
on $\Gamma _{R(Q_i,\ldots,Q_n)}$ is a canonical automorphism on
$\Gamma _{R(Q_i,\ldots,Q_n)}$,  identical on variables from
$X_{i+1},\ldots ,X_n$ and on all free variables $t_i$ from the higher
levels, $\pi _i:\Gamma _{R(Q_i,\ldots,Q_n)}*F(t_1,\ldots
,t_{k_{i-1}})\rightarrow \Gamma _{R(Q_{i+1},\ldots,Q_n)}*F(t_1,\ldots
,t_{k_i}).$ The {\em dimension of the fundamental sequence} is the sum $k_1+\ldots +k_n.$

We only consider strict fundamental sequences. Recall, that, in particular, this means that we include in the fundamental sequence only such homomorphisms that
give nonabelian images of the   regular subsystems of ${Q_i}=1$ on
all  levels (the rest can be included into a finite number of
fundamental sequences), and the images
of the edge groups of the JSJ decompositions on all levels are
nontrivial.

We can  suppose that all fundamental sequences that we consider
satisfy the following properties. Let $\Gamma _{R(Q_i,\ldots ,Q_n)}$ be a
free product of some factors. Then
  \begin {enumerate}
\item [1)] the images of abelian factors   under $\pi _i$  are
different factors of $F(t_{k_{i-1}+1},\ldots ,t_{k_i})$; \item [2)]
the images under $\pi _i$ of factors which are surface groups are
different  factors of $F(t_{k_{i-1}+1},\ldots ,t_{k_i})$ or of $\Gamma$;
\item [3)] if some quadratic equation in $Q_i=1$ has free
variables in this fundamental sequence, then these variables
correspond to some variables among $t_{k_{i-1}+1},\ldots ,t_{k_i}$,
the images under $\pi _i$ of coefficients of quadratic equations
cannot be conjugated into $F(t_{k_{i-1}+1},\ldots ,t_{k_i})$; \item
[4)] different  factors in the free decomposition of
$\Gamma _{R(Q_i,\ldots ,Q_n)}$ are sent into different factors in the free
decomposition of $\Gamma _{R(Q_{i+1},\ldots
,Q_n)}*F(t_{k_{i-1}+1},\ldots,t_{k_i}).$
\end{enumerate}

\begin{prop}
\label{pr:solutions-special-Sb}  For a system of equations $S$ over $\Gamma$ one can effectively construct a
finite set of  strict completed fundamental sequences  that  satisfy  the first and second
restrictions above, such that every solution of $S$ factors through one of these fundamental sequences. These fundamental sequences correspond to a tree which we denote $T_{CE}(\Gamma _{R(S)})$ (canonical embedding tree).
\end{prop}
\begin{proof} Using Proposition \ref{11} we can construct a finite number of strict (completed) fundamental sequences over $\Gamma$ such that each solution of $S$ factors through one of them. In addition, if the corresponding fundamental sequence over a free group satisfies first and second restrictions, then the fundamental sequence obtained in Proposition \ref{11} can be also constructed  satisfying these restrictions. To do this algorithmically one only needs to solve the word problem in $\Gamma$-limit groups which are given as subgroups of NTQ groups. And the word problem in NTQ groups is solvable because they are toral relatively hyperbolic.\end{proof}

\begin{definition} We call  fundamental sequences satisfying the first and second restrictions {\em well aligned} fundamental sequences.\end{definition}

\subsection{Induced NTQ systems and fundamental sequences}\label{inher}

In this subsection we modify the construction of \cite{KMel}, Section 7.12 for a torsion free hyperbolic group $\Gamma$. Given an NTQ system $S=1$ over $\Gamma$, the corresponding NTQ group $\Gamma _{R(S)}$, and the well aligned completed fundamental sequence of solutions, we will construct the induced NTQ group, the NTQ system, and the well aligned completed fundamental sequence of solutions for a subgroup $K$ of $\Gamma _{R(S)}$.

Let $S=1$ be an NTQ system over $\Gamma$:

\medskip
$S_1(X_1, X_2, \ldots, X_n,A) = 1,$

\medskip
$\ \ \ \ \ S_2(X_2, \ldots, X_n,A) = 1,$

$\ \ \ \ \ \ \ \ \ \  \ldots$

\medskip
$\ \ \ \ \ \ \ \ \ \ \ \ \ \ \ \ S_n(X_n,A) = 1$

\medskip \noindent
and $\pi _i: \Gamma_i \rightarrow \Gamma_{i+1}$   a fixed
$\Gamma_{i+1}$-homomorphism (a solution of $S_i(X_1,\ldots ,X_n)=1$ in
$\Gamma_{i+1} = {\Gamma}_{R(S_{i+1},\ldots ,S_n)},$ $\Gamma_{n+1} = G$, which is a free product of $\Gamma$ and a  free group). 
Let $K$ be a finitely generated subgroup (or $\Gamma$-subgroup) of
$\Gamma_{R(S)}.$ Then there exists a system $W(Y) = 1$ such that  $ K =
\Gamma_{R(W)}.$  We will describe here how to embed $K$ more economically into an
NTQ group $\Gamma_{R(Q)}$ such that ${\Gamma_{R(Q)}}\leq _{quasi\ isom.}\Gamma_{R(S)}$ and assign to
$Q=1$ a fundamental sequence that includes all the solutions of
$W=1$ that factor through the  well aligned  fundamental sequence for $S=1$ that we started with.

Canonical automorphisms of $\Gamma _{R(Q_i,\ldots Q_n)}$ on different levels for $Q=1$ will be
induced by canonical automorphisms for $S=1$, mappings between
different levels for $Q=1$ will be induced by mappings for $S=1.$
 
 Without loss of generality we can suppose
that $\Gamma_{R(S)}$ is freely indecomposable modulo $\Gamma$.  
 The top
quadratic system of equations $S_1(X_1,\ldots ,X_n)=1$ corresponds
to a splitting $D$ of $\Gamma_{R(S)}$. The non-QH non-abelian vertex subgroups
of $D$ are factors in a free decomposition of $\langle X_2,\ldots ,X_n\rangle .$
Consider the induced splitting of  $K$ denoted by $D_K$. This
splitting may give a free factorization $K= K_1*\ldots *K_k,$ where
$\Gamma\leq  K_1$. Consider each factor separately.  Consider $K_1$. Each edge $e$ in the decomposition of $K_1$ (induced by $D_K$ on the free factor $K_1$) that connects two rigid vertex groups is composed from  two edges $e_1$ and $e_2$ that are adjacent and both are in the orbit of the same edge $\bar e$ in the Bass-Serre tree corresponding to the graph of groups $D$. Moreover, $\bar e$ connects a rigid vertex group to an abelian vertex group in $D$, because if it connects 
a rigid vertex group to a QH subgroup, then  a QH subgroup would appear between two rigid vertex groups instead of edge $e$. Increasing $ K_1$ by
a finite number of suitable elements from abelian  vertex groups of
$\Gamma_{R(S)}$ we join together non-QH non-abelian subgroups of $D_K$
which are conjugated into the same non-QH non-abelian subgroup of
$D$ by elements from abelian vertex groups in $\Gamma_{R(S)}.$ Moreover, we can choose these elements $a$ in abelian subgroups in such a way that their images on the next level are trivial. Indeed, each abelian subgroup $A$ is the direct product of the isolator of the edge group $A_1$ (in a primary decomposition this isolator is $A_1$ itself) and a free abelian subgroup $A_2$. This conjugating element $a$ belongs to $A_2$ and, therefore, is mapped on the next step to (the image of) $A_1$, say $a_1\in A_1$. Then $a_1^{-1}a$ is the desired element that is mapped to the identity. When we add this element to $K$ instead of the  two non-QH non-abelian subgroups in $D_k$ considered above we have one non-QH non-abelian subgroup. This way we obtain a group $\bar K_1$ such that $D_{\bar K_1}$ does not
have edges between non-QH non-abelian subgroups, and generators of
edge groups connecting non-QH non-abelian subgroups to abelian
subgroups not having roots in $\langle X_2,\ldots ,X_N\rangle .$  We do the same if an abelian vertex group is connected to two rigid subgroups. Denote the group that we obtained by $\hat K_1$.

Since we will be considering only well aligned fundamental sequences,  we fix the family of s.c.c. in the QH subgroups of $\hat K_1$ that are mapped to the identity by $\pi _1$, so that the corresponding quadratic equations split into systems satisfying  the first restriction.  (We will identify s.c.c. on the surface with corresponding elements in the QH subgroup.) This corresponds to a collection of s.c.c. which are the pre-images of the collection of s.c.c. in
a QH subgroup of $\Gamma_{R(S)}.$ We will add conjugating elements so that each punctured surface connected to a rigid subgroup in the  decomposition refined by splittings along these s.c.c, is connected to a unique rigid subgroup.  Now all rigid subgroups that are mapped to the same free factor on the next level of ${\Gamma}_{R(S)}$, are conjugate into one subgroup.
The images of these elements in $\Gamma _{R(S)}$ are mapped to the identity by $\pi _1$. Denote the obtained group by $\tilde K_1$.  We have  $\pi _1(K)=\pi _1(\tilde K_1).$

Now
consider separately each factor in the free decomposition of $\pi _1(K)\cap \langle X_2,\ldots
,X_n\rangle $ and enlarge it the same way. Working similarly with each
$\Gamma_i$ we consider all the levels of $S=1$ from the top to the bottom. 
We  obtain a subgroup generated by $\tilde K_1,\ldots , \tilde K_k$ and the enlarged images of them on all the levels, and   denote it by $H_1$. 

Then we repeat the whole
construction for $H_1$ in place of $K$, obtain $H_2$ and repeat the
construction again. We will eventually stop, namely obtain that
$H_i=H_{i+1}$, because every time when we repeat the construction if $H_i\neq H_{i+1}$ then there is some level $j$ such that on all the levels higher than $j$ the decompositions are the same as in the previous step and on level $j$ 
one of the following characteristics decreases:
\begin{enumerate}\item the number of free factors in the free
decomposition on  level $j$ of $H_i,$  \item if the number of free
factors does not decrease, then the number of edges and vertices of
the induced decomposition on  level $j$ of $H_i$ decreases,
\item if the number of free factors in the free decomposition on
 level $j$ of $H_i$  and the number of edges and vertices does not
decrease, then the size of the decomposition of level $j$ of $H_i$ is decreased.
\end{enumerate}
 If there are  freely indecomposable factors  of $K^{\pi _1\ldots \pi _n}$ that are isomorphic to subgroups of $\Gamma$, then we treat them the same way as we treated such factors constructing canonical fundamental sequences. Namely we add the level where there is a free product of these factors and on the following level we have a fixed embedding of each factor into a conjugate of $\Gamma$. Extend it to its  quasi-convex closure, and complete to the NTQ group as in Section \ref{complete}.
The corresponding NTQ system is
$Q=1$ such that $K\leq \Gamma_{R(Q)}\leq _{quasi\  isom.}\Gamma_{R(S)}$.  Each QH subgroup of $\Gamma_{R(Q)}$
is a finite index subgroup in some QH subgroup of  $\Gamma_{R(S)}$. Canonical automorphisms corresponding to QH subgroups of $\Gamma _{R(Q)}$ well be induced by canonical automorphisms of $QH$ subgroups of $\Gamma _{R(S)}$. The image of the
top $j$ levels of $\Gamma_{R(Q)}$ on the level $j+1$ is contained in  the
image of $K$ on this level. It is, actually, $K^{\pi _1\ldots \pi_j}$.   This NTQ system $Q=1$ is called the {\em induced NTQ system } and the corresponding well aligned  fundamental sequence is called the {\em induced fundamental sequence}. 

Similarly, if we have a fundamental sequence (and an NTQ system) modulo
a subgroup we  can define the induced fundamental sequence (and the NTQ
system)  modulo a subgroup $H$.

There is an algorithm for this  construction over a free group, that can be used to construct the induced completed fundamental sequence over $\Gamma$ as in \cite{KMT}, Proposition 9.  The terminal $\Gamma$-limit  group of  such induced completed fundamental sequence  is given by its generators inside a $\Gamma$-NTQ group. Actually,  for the decidability of the $\forall\exists$-theory of $\Gamma$ we do not need this algorithm,  we only need  the existence of induced fundamental sequences.

\subsection{First step}\label{aex}
We will now describe the algorithm for construction of the $\forall\exists$-tree.
Consider the sentence
\begin{equation}\label{ae}
\Phi=\forall  X\exists Y (U(X,Y)=1\wedge V(X,Y)\neq 1)\end{equation}

If the sentence is true then  there exists a solution of
$ U(X,Y)=1\wedge V(X,Y)\neq 1$ in $F(X)*\Gamma$. By (\cite{D1}, Theorem 0.1) there is an algorithm to find such a solution $Y=f(X)$. To check whether the sentence (\ref{ae}) is true we now have to check it only for  those values of $X$ for which  $V(X,f(X))=1.$  Denote $U_0(X)=V(X,f(X)).$


 Let $G=\Gamma _{R(U_0)}$.
We define now a tree
$T_{EA}(\Phi)=T_{EA}(G)$ oriented from the root, and assign to each vertex of
$T_{EA}(G)$ some set of homomorphisms from $G$ to $\Gamma$. We assign the set of all homomorphisms  $G\rightarrow \Gamma$ to the initial vertex $\hat v_{-1}$.
 We can
construct algorithmically a finite number of $\Gamma$-NTQ systems  corresponding to branches $b$ of the canonical $Hom$-diagram described in Proposition \ref{11},   (denote the system corresponding to the branch $b$ by $S(b)$), their corrective extensions: $S_{corr}(b)=1$
($S_{corr}(b):S_1(X_1,\ldots ,X_n)=1,\ldots ,S_n(X_n)=1$), and corresponding  well aligned
fundamental sequences $Var_{\rm fund}(S(b))$. For each such fundamental
sequence we assign a vertex $\hat v_{0i}$ of the tree $T_{EA}(G)$.
We draw an edge from
 vertex $\hat v_{-1}$  to each vertex corresponding to $Var_{\rm fund}(S(b)).$

 Let the fundamental sequence 
$Var_{\rm fund}(S(b))$ be assigned to $\hat v_{0i}$. Since every branch of the tree will be constructed independently of the others we will now describe the construction for $Var_{\rm fund}(S(b))$.  By Theorem \ref{IFT} and by Proposition \ref{relh}, we can algorithmically find
values of $Y$ given by  formulas $Y=f(X_1,\ldots ,X_n)$ in 
$X_1,\ldots ,X_n$ taking values in  ${\Gamma}_{R(S(b))}$.  Indeed, ${\Gamma}_{R(S(b))}$ is and its corrective extensions are toral relatively hyperbolic, and the roots that have to be added to ${\Gamma}_{R(S(b))}$ to obtain a corrective extension can be found algorithmically using Dahmani's construction of canonical representatives for toral relatively hyperbolic groups \cite{D1}. 

Sentence (\ref{ae}) is now verified for all values of $X$ in $Var_{\rm fund}(S_{corr}(b))$ except those for which 
  we have $V(X_1,\ldots ,X _n, f(X_1,\ldots ,X_n))=1.$  As on the previous step this gives an equation
$U_1(X_1,\ldots ,X_n)=1$ which is a disjunction of irreducible equations, so we assume it is irreducible. 
We will only consider values of $X_1,\ldots ,X_n$ satisfying this system that are minimal in $Var_{fund}(S(b))$ with respect to canonical automorphisms on all the levels modulo the image of $G$  in $S(b)$.  This is enough because if for a specialization $X^{\phi}$ there exists a minimal specialization of $X_1^{\phi},\ldots ,X_n^{\phi}$ that does not satisfy this system, then the sentence is true for $X^{\phi}$. This assumption implies another equation on   $X_1,\ldots ,X_n$ (that can be found algorithmically by Proposition \ref{12}). We assume  that we included this other equation  into $U_1(X_1,\ldots ,X_n)=1$. 

Let $Var_{\rm
fund}(U_1)$ be the subset of homomorphisms from the set $Var_{\rm
fund}(S(b))$ going through the corrective extension 
$S_{corr}(b)=1$, minimal with respect to the canonical automorphisms modulo the image of $G$ on all levels, and satisfying the additional equation
$U_1(X_1,\ldots ,X_n)=1$, and embedding the conjugates of subgroups of $\Gamma$ in $G$. We introduce a vertex $\hat v_{1i}$ for each such equation and
draw an edge connecting $\hat v_{0i}$ to $\hat v_{1i}$. 

Let $K$ be a finitely generated group. Recall that any family of homomorphisms $\Psi=\{\psi _i:K\rightarrow \Gamma \}$ factors through a finite set of maximal fully residually $\Gamma$ groups $H_1,\ldots , H_k$ (= $\Gamma$-limit groups) that all are quotients  of $K$.  We first take a quotient $K_1$ of $K$ by the intersection of the kernels of all homomorphisms from $\Psi$, and then construct maximal fully residually $\Gamma$ quotients $H_1,\ldots , H_k$ of $K_1$. We say that $\Psi$ {\em discriminates} groups  $H_1,\ldots , H_k$, and that each $H_i$ is {\em a fully residually $\Gamma$ group discriminated by $\Psi .$}

 Let $G_1$ be a fully residually $\Gamma$ group discriminated by the set of homomorphisms
$Var_{\rm fund}(U_1)$ (we do not need to know effectively its presentation). Consider the family  of  fundamental
sequences for $G_1$  modulo the images $R_1,\ldots, R_s$ of the factors in the free  decomposition of the subgroup $H_1=\langle X_2,\ldots, X_n\rangle $. (We say that a fundamental sequence is constructed modulo some subgroups of the coordinate group  if these subgroups are elliptic in the JSJ decompositions on all levels in the construction of this fundamental sequence.) We know the generators and relations of  $R_1,\ldots, R_s$ and by Proposition \ref{12} can effectively construct these fundamental sequences.
 We only consider well aligned canonical fundamental sequences $c$ for $G_1$ modulo the images of $R_1,\ldots, R_s$  (corresponding to coefficients of quadratic
 equations $S_1=1$ of the top level for $S_{corr}(b)=1$) with, in particular, the following properties: 
 
 (1) They have dimension less than or equal to
 $k_1$,  
 
 (2) The  edge groups for edges connected to $R_1,\ldots, R_s$ are not mapped to the identity,  
 
 (3) The homomorphisms embed the images of the terminal non-cyclic freely indecomposable factors of $Var _{fund}(S(b))$ into $\Gamma$, 
 
 (4)  For each QH vertex group $Q$  in the abelian decomposition of the top level of $F_{R(S_{corr})},$ the boundary elements of $Q$ are congugate into the non-QH non abelian subgroups of the top level of $F_{R(S_{corr})}.$ Hence  the images  of  $Q$  have cyclic decompositions induced by decompositions of different levels of $c$. These decompositions have to be compatible with  the decompositions of surfaces corresponding to
 quadratic equations of $S_1$,  by a collection of simple closed curves mapped to the identity. Namely, if we refine
 the JSJ decomposition of $G$ by adding splittings
corresponding to the simple closed curves that are mapped to the identity when $G$ is mapped to the free product in Subsection \ref{spl}, then the
standard coefficients on all the levels of $c$ are images of
elliptic elements in this decomposition.

Suppose a fundamental sequence $c$ has the top dimension component
$k_1$ (Since we only consider well aligned fundamental sequences, we would not consider $c$ if it had the top dimension greater than $k_1$). If the NTQ system corresponding to the top level of the sequence
$c$ is the same as $S_1=1$, we extend the fundamental sequences
modulo $R_1,\ldots ,R_s$ by canonical fundamental sequences  for
$H_1$  modulo  the factors in the free decomposition of the subgroup
$\langle X_3,\ldots ,X_n\rangle $. If such a sequence has dimension greater than
or equal to $k_2$, then the corresponding solution can be factored
through a fundamental sequence for $U=1$ of
 the greater dimension.
Again, we only consider such sequences of dimension less than or equal
to $k_2$. If the sum of the first two dimensions is strictly smaller
than $k_1+k_2$, we do the same as in the case when the first dimension
is smaller than $k_1$ (see below). We continue this way to construct
fundamental sequences $Var_{\rm fund}(S_1(b))$. We draw edges of the tree $T_{EA}(G)$ from the
vertex corresponding to $Var_{\rm fund}(U_1)$ to the vertices $Var_{\rm
fund}(S_1(b))$.

 Suppose now that the fundamental sequence $c$ for $G_1$ modulo
$R_1,\ldots ,R_s$ has dimension strictly less than $k_1$ or has
dimension $k_1$, but the NTQ system corresponding to the top level of
$c$ is not the same as $S_1=1$. Suppose also that $G\neq \Gamma$. Then we
use the following lemma (in which we suppose that $R_1,\ldots ,R_s$
are non-trivial).

\begin{lemma}\label{prop}The image $G_{t}$  of $G$ in
the group $H_{t}$ appearing on the terminal level $t$ of the sequence $c$ is a
proper quotient of $G$ unless each of the homomorphisms that factor through $H_t$ is an embedding of each of freely indecomposable factors of $H_t$ into $\Gamma$.
\end{lemma}
\begin{proof} Consider the terminal group of $c$; denote it $H_t$.
Suppose $(G)_{t}$ is isomorphic to $G$. Denote the abelian JSJ
decomposition of $H_t$ by $D_t$ (we mean the free decomposition and then decompositions of free factors). Then there is an abelian
decomposition of $G$ induced by $D_t$. Therefore rigid (non-abelian
and non-QH) subgroups and edge groups of
$G$ are elliptic in this decomposition.

By Proposition \ref{10}, there exists a decomposition of some free factor $P_i$ of $H_t$ which is induced from
$D_t$ (because not all such factors are embedded into $\Gamma$).  But this is impossible because this  means that
the homomorphisms we are considering can be shortened by applying
canonical
automorphisms of $P_i$ modulo those subgroups from $\{R_1,\ldots ,R_s\}$ which are conjugated into $P_i$ (since $Var _{fund}(U_1)$  contains only homomorphisms minimal with respect to canonical automorphisms modulo the image of $G$ on all levels, these subgroups must be also elliptic in $D_t$).

\end{proof}

Often we may get a proper quotient of $G$ on some level of $c$ that is not a terminal level. For example, if the Kurosh rank (the number of freely indecomposable factors in the Grushko decomposition plus the rank of the free group) on some level $l$  is larger than the Kurosh rank of $G$, then the natural image $G_l$  is a proper quotient of $G$.  We can effectively find the set of completed fundamental sequences for the terminal group of $c$ and induced (from these sequences) fundamental sequences and NTQ groups  for $G_t$ (or $G_l$ if we found such a level $l$ algorithmically). We consider   fundamental sequences   for these groups (we include all the automorphisms of QH subgroups in these fundamental sequences, not only the induced ones). We do not include those fundamental sequences where  we obtained proper quotients of subgroups of $\Gamma$ (the corresponding homomorphisms factor through other fundamental sequences assigned to other vertices $\hat v_{0,i}$ of the $T_{EA}(G)$).   Denote the complete set  of  these fundamental sequences  $G_t\rightarrow \Gamma$ by $\mathcal F$. One can extract from $c$ modulo the terminal
level  the induced  well aligned fundamental sequence for $G$.  Denote this induced fundamental sequence  by $c_2$. We consider a fundamental sequence $c_3$ that consists of homomorphisms obtained by the composition
  of a homomorphism from $c_2$  and a homomorphism from a fundamental sequence
  corresponding to one of the branches $b_2$  int he family $\mathcal F$. Consider the {\em block-NTQ group}
 $\bar G$  generated by all the  levels of the NTQ group corresponding
 to the fundamental sequence $c$ (or up to level $l$ if this is the level where we detected a proper quotient of $G$), and  the NTQ group  corresponding
 the fundamental sequence  $c_3$  and amalgamated along the common part.   We can construct $\bar G$ algorithmically using \cite{KMT}, Proposition 6.  There will be, actually several such groups, because $G$ can be  mapped to the terminal level by a natural projection  but also by a natural projection shifted by automorphisms on each level of $c$
(there can only be a finite number of such shifted induced fundamental sequences). 
  
One can apply Theorem \ref{IFT} to the NTQ group $N$ corresponding to $c_3$ and get formula solutions of $V(X,Y)=1$ in a corrective extension of $\bar G$ (almost as in \cite{Imp}, Section 7.5, except that the formula solution exists in some extension of $N$ which is not just a corrective extension, but may have  QH subgroups of $N$  extended  to finite index subgroups of the corresponding QH subgroups of $c$).    Assign a vertex $\hat v_{2ik}$ of the tree $T_{EA}(G)$ to sequence $c_3$.
 We draw an edge from the vertex $\hat v_{1i}$ of $T_{EA}(G)$ corresponding to $Var_{fund}(U_1)$ to $\hat v_{2ik}$.  All the vertices  $\hat v_{2ik}$ are coming out of  $\hat v_{1i}.$ By Theorem \ref{IFT} we can find formula solutions of $U(X,Y)=1$ over this extension of $\bar G$. Those formula solutions for which $V(X,Y)=1$ 
 (if exist) will give an additional equation $U_2=1$ for generators of this extension of $\bar G$.

\subsection{Second step}\label{2nd}
We will describe the next  step in the construction of $T_{EA}(G )$
which basically is general. Fundamental sequences and block-NTQ groups obtained on the second step will be assigned to vertices $\hat v_{3jks}$ of the tree $T_{EA}(G)$. 

 Let $c_3$ be a
fundamental sequence corresponding to some vertex $\hat v_{2ik}$
 of $T_{EA}(G)$, let $c$ be, as before, the corresponding canonical fundamental sequence for $G_1$
 modulo $R_1,\ldots ,R_s$. Consider the set of those minimal homomorphisms from $\bar G$ to $\Gamma$
 which are going through the fundamental sequence $c_3$, which  factor through a corrective extension,  and satisfy the additional equation $U_2=1$.
 Let $G_2$ be one of the fully residually $\Gamma$ groups discriminated by this set.

 Let  $G^{(2)}$ be the image of $G$ in $G_2$.
The family of solutions with which we continue satisfies the first and second restrictions, and this property can be verified algorithmically.

Suppose  the JSJ decomposition for the NTQ system corresponding to the top level of $c$ corresponds to
the equation $S_{11}(X_{11},X_{12},\ldots )=1;$ some of the
variables $X_{11}$ are quadratic, the others correspond to
extensions of centralizers.
 Construct a fundamental sequence $c^{(2)}$ as in Proposition \ref{prop7} for $G_2$ modulo  the factors in the free decomposition of the subgroup generated by
$X_{12},\ldots  $.

Suppose $c$ is not the same as the top level of $S(b)$.  Denote by $N_0^1$ the image of the subgroup generated by $X_1,\ldots
,X_n$ in the group discriminated by $c$. So, $N_0^1= \langle X_1,\ldots
,X_n\rangle _c.$ Denote by $N_0^2=\langle X_1,\ldots ,X_n\rangle _{c^{(2)}}$ the image of
$\langle X_1,\ldots ,X_n\rangle $
 in the group discriminated by $c^{(2)}$. If the images of  some edge groups of $N_0$ are trivial in $N_0^2$ or the images of some non-abelian vertex groups are abelian (we have an algorithm to check this), then $N_0^2$ is a proper quotient of $N_0^1$, and  we have another ``easy'' case.
 In this case we take the quotient of $N_0$ adding these collapsing relations  of $N_0^2$,
instead of $G_2$, and assign its fundamental sequences to $\hat v_{3jks}$, and we do not consider vertices corresponding to NTQ systems with the same top level as the NTQ system for $c_3$.
 
In all the cases below we suppose that there are no collapses of $N_0$ in $N_0^2$.

{\bf Case 1.} If the top levels of $c$ and $c^{(2)}$ are the same, then
we go to the second level of $c$ and consider it the same way as the
first level.

{\bf Case 2.}  If the top levels of the NTQ system for $c$ and $S_1$ are the same
(therefore $c$ has only one level). We work with $c^{(2)}$ the same
way as we did for $c$. Then  the image of  $G$  on the terminal  level, say,  
$k$ of $c^{(2)}$ is a proper quotient of $G$ by
Lemma \ref{prop}. If at some point the sum of dimensions for
$c^{(2)}$ is not maximal, we amalgamate the fundamental sequence induced by the top part of $c^{(2)}$ 
above level $k$ for $G$, and each fundamental sequence induced by this quotient (solutions will go along the first fundamental sequence from the top level to level $k$ and then continue along one of the fundamental sequences induced by  the group on level $k$, with non-reduced canonical groups of automorphisms for QH subgroups).  We assign each of these fundamental sequences to a vertex  $\hat v_{3jks}$, then take the family of corrective extensions. Then we construct the block-NTQ group as we did on the first step, denote it by $N_2$. We also assign $N_2$  to the vertex  $\hat v_{3jks}$.

{\bf Case 3.}  If the top levels of the NTQ system for $c$ and $S_1$ are  not  the same and
the top levels of $c$ and $c^{(2)}$ are not the same, then we look at  $N_0^2$ and $N_0^1$.  

Then we can suppose that the image $N_{0,t}^1$ of
 $N_0^1$  on the terminal level of $c^{(2)}$  
 is a proper quotient.  Suppose this terminal level is level $k$.
  
 Consider
fundamental sequences  induced by $N_{0,t}^1$ modulo the images of subgroups $R_1,\ldots, R_s$  with non-reduced, canonical groups of automorphisms for QH subgroups, and apply to them step 1. Denote the obtained fundamental sequences  by $f_i$.
Construct fundamental sequences for the subgroup generated by the images of 
$X_1,\ldots ,X_n$ with the top part being induced from the top
part of $c^{(2)}$ (above level k) and bottom part being some $f_i$,
but not the sequence with the same top part as $c$. We construct a block-NTQ group amalgamating the top $k-1$ levels of $c^{(2)}$ and the block-NTQ group constructed for $f_i$ as on the first step. There exists a formula solution over the corrective extension of this group. 

If all the levels of $c$ and $c^{(2)}$ are the same, so we never have cases 2 and 3,  and the same happens on all levels of the block-NTQ group $\bar G$ (see step 1), then 
the block fundamental sequence consists of a sequence of induced fundamental sequences
for $G$ and its images. Denote it $\overline c_2$. For each level of $\overline c_2$ there is an abelian decomposition. 
 Denote by $G_{corr}$ the corrective extension of the group corresponding to $\overline c_2$.  Denote  the fundamental sequence induced from $c$ and its continuation  by $c_4$.  There exists some level $k$ such that the abelian decompositions for the NTQ group  for $c_4$ will coincide with abelian decompositions for  the NTQ group for $\overline c_2$  for levels above $k$, and on level $k$ either the number of free factors in the free decomposition  for $c_4$ is less than this number for   $\overline c_2$, or the number of factors is the same, but the  regular size of the decompositions (lexicographically ordered tuple $(size (Q_1),\ldots, size (Q_m))$  of sizes of  MQH subgroups)  for $c_4$ is  smaller  than that  for $\overline c_2$, or the regular sizes are the same but the abelian size  of the decompositions  for $c_4$ is  smaller  than that for $\overline c_2$.  Here, if $R$ is an abelian decomposition,
by $ab(R)$ (the abelian size) we denote the sum of the ranks of abelian vertex groups
in $R$ minus the sum of the ranks of the edge groups for the edges
from them.

We will take $c_4$ to the next step instead of $\overline c_2$.

\subsection{General step}\label{nth}
We now describe  the $n$'th step of the construction. Denote by
$N_i$ the block-NTQ group constructed on the $i$'th step, and by
$N_i^j, j> i$ its image  on the $j$'th step. 
Fundamental sequences and block-NTQ groups obtained on the step $n-1$ are assigned to vertices $\hat v_{njks\ldots }$ of the tree $T_{EA}(G)$. 

Let $\{j_k,\
k=1,\ldots ,s\}$ be all the indices for which the top level of
$N_{j_k+1}$ is different from the top level of $N_{j_k}$.

If  some of the groups  $N_{j_k}^{j_k+1}$  have collapses 
in $N_{n-1}^n$ (we can check this the same way as we checked whether the image of $N_0^1$ has collapses Step 2,  Case 3), then we replace the first such group by its proper
quotient in $N_{n-1}^n$ and consider only the fundamental sequences
that have the top level different from $N_{j_k}.$ In all other cases
we can suppose that  $G$ and all the groups $ N_{j_k}^{j_k+1}$ do not have any collapses  in $N_{n-1}^n.$

{\bf Case 1.} The top levels of $c^{(n)}$ and $c^{(n-1)}$ are the
same. In this case we go to the second level and consider it the
same way as the first level.

If going from the top to the bottom of the block-NTQ system, we do
not obtain the case considered above or Cases 2, 3 and all the levels of the top
block of $N_{n-1}$ and $N_n$ are the same, we  
consider the group $G_{corr}$ which was constructed for the induced
fundamental sequence corresponding to the homomorphisms from $G$
going through $N_{n}$, and the fundamental sequence induced by $c^{(n)}$ for this group as we did on the second step when cases 2 and 3 were not applicable on all the levels of $c^{(2)}$.

{\bf Case 2.} The top levels of $c^{(n-1)}, c^{(n-2)},\ldots ,
c^{(n-i)}$ are the same, and the top levels of $c^{(n-1)}$ and $c^{(n)}$
are not the same. Then on the terminal level $p$ of the NTQ group for $c^{(n)}$ we can suppose
that the image of $N_{n-i-1}^{n-i}$ is a proper quotient,   or the
fundamental sequence goes through another branch constructed on the
previous step. Consider fundamental sequences  $f_i$ induced by this
quotient modulo its rigid subgroups  (with non-refuced canonical groups of automorphisms corresponding to QH subgroups) and apply to them  the procedure described on Step 1 . Consider only sequences with the top level different from $c^{(n-i-1)}$. We construct
$N_n$ as a block-NTQ group with the top part being the NTQ group for $c^{(n)}$ above
level $p$ and the bottom part being the block-NTQ group corresponding to $f_i$.

{\bf Case 3.} The top levels of $c^{(n-2)}$ and $c^{(n-1)}$ are not
the same and the top levels of $c^{(n-1)}$ and $c^{(n)}$ are not the
same. Then on the terminal  level $p$ of $c^{(n)}$ the image of
$N_{n-2}^{n-1}$ is a proper quotient.
Construct a block-NTQ group as in the previous case.

In this way we continue the construction of the tree $T_{EA}(G)$.

\subsection{The $\forall\exists$ tree is finite}\label{5.6}

It is convenient to define as in \cite{Sela}, Definition 4.2,  the notion of complexity of a fundamental sequence ($Cmplx (Var_{fund})$) at follows:
$$Cmplx(Var_{\rm fund}) =(dim (Var_{\rm fund}) + factors (Var_{\rm fund}), (size (Q_1),\ldots ,size (Q_m)), ab(Var _{\rm fund}(Q));$$
where \begin{itemize} \item $dim (Var_{\rm fund})$ is the rank of the free group on the terminal level; \item $factors (Var_{\rm fund})$ is the number of freely indecomposable, non-cyclic terminal factors embedded into $\Gamma$, \item $(size (Q_1),\ldots, size (Q_m))$ is the regular size of the system (lexicographically ordered tuple $(size (Q_1),\ldots, size (Q_m))$  of sizes of  all MQH subgroups that appear in the corresponding (block-)NTQ group); \item $ab(Var _{\rm fund}(Q))$  is the abelian size of the corresponding NTQ system (defined after the definition of corrective extensions), by other words it is  the sum of ranks of the kernels of the mappings of abelian groups that appear as vertex groups along the fundamental sequence.\end{itemize}
The complexity is a tuple of numbers which we compare in the left lexicographic order. (Sela calls
 $dim (Var_{\rm fund}) + factors (Var_{\rm fund})$ the Kurosh rank of the resolution)

In this subsection we will prove the following result.

\begin{theorem} \label{AE7} The tree $T_{EA}(G)$ is finite.\end{theorem}
Proof.
 Let an NTQ system $Q(X_1,\ldots, X_n)=1$ have the form
$$Q_1(X_1,\ldots ,X_n)=1,$$
$$\ldots $$
$$Q_n(X_n)=1.$$
 Denote by $D_Q$ a canonical decomposition corresponding to the group
$\Gamma _{R(Q)}$. Non-QH, non-abelian subgroups in this decomposition are
$P_1,\ldots ,P_s$. Abelian and QH subgroups correspond to the system
$Q_1(X_1,\ldots ,X_n)=1,$ variables from $X_1$ are either quadratic, or correspond
to abelian vertex groups. Consider the system $Q(X_1,\ldots, X_n)=1$ together with the 
fundamental sequence $Var_{\rm fund}(Q)$ defining it. Let $Var_{\rm
fund}(U_1)$ be the subset of $Var_{\rm fund}(Q)$ satisfying some
additional equation $U_1=1$, and $G_1$ a group discriminated by this
subset. Consider the family  of those canonical fundamental
sequences for $G_1$ modulo the images $R_1,\ldots ,R_s$ of the
factors $P_1,\ldots ,P_s$ in the free decomposition $H_1*=P_1*\ldots *P_s$ of
 the subgroup $\langle X_2,\ldots ,X_m\rangle $, which have the same Kurosh rank modulo them as
 $Q_1=1$.  This means that these sequences are compatible with the splitting of quadratic equations according to item 6) in  the first restriction on fundamental sequences with $Var_{fund}(Q).$ 
Constructing this fundamental sequence we take into consideration only those homomorphisms of the quotient $\Gamma$-limit group $G_1$ that embed the images of the terminal non-cyclic freely indecomposable factors into $\Gamma$.
Denote such a fundamental sequence by $c$, and the corresponding NTQ
system $S=1 (mod \ H_1*)$, where $S=1$ has the form
$$S_1(X_{11},\ldots ,X_{1m})=1$$
$$\ldots $$
$$S_m(X_{1m})=1.$$

 For each $i$ there exists a canonical
homomorphism
$$\eta _i: \Gamma _{R(Q)}\rightarrow \Gamma _{R(S_i,\ldots ,S_m)}$$
such that $P_1,\ldots ,P_s$ are mapped into
rigid subgroups in the canonical decomposition of $\eta
_i(\Gamma _{R(Q)})$.

Each QH subgroup in the decomposition  of $\Gamma _{R(S_i,\ldots ,S_m)}$
as an NTQ group is a QH subgroup of $\eta _i(\Gamma _{R(Q)})$. By \cite{Sela4}, Lemma 2.7,
for each QH subgroup $\hat Q$ of $\eta
_i(\Gamma _{R(Q)})$ there exists a QH subgroup of $\Gamma _{R(Q)}$ that is
mapped onto a subgroup of finite index in $\hat Q$. The size of this QH
subgroup is, obviously, greater or equal to the size of $\hat Q$. Those
MQH subgroups of $\Gamma _{R(Q)}$ that are mapped into QH subgroups of the
same size by some $\eta _i$ are called {\em stable}.

\begin{lemma}\label{ch}
In the conditions above there are the following possibilities:

(i) The set of homomorphisms going through $c$ is generic for each
regular quadratic equation in $Q_1=1$ and $ab(c)=ab(Var_{\rm
fund}(Q))$ (in this case $c$ has only one level identical to $Q_1$);

(ii) It is possible to reconstruct system $S=1$ in such a way that $size(S) < 
size(Q_1)$;

(iii)  $size(S) = size(Q_1)$, $ab(c)< ab(Var_{\rm fund}(Q)).$

\end{lemma}

\begin{proof}

 The fundamental sequence $c$ modulo the decomposition $H_1* $
has the same dimension as $Q_1=1$. The Kurosh rank of $Q_1=1$ is the
sum of the following  numbers: \begin{enumerate} \item [1)] the
dimension of a free factor $F_1=F(t_0,\ldots ,t_{k_0})$ in the free
decomposition of $F_{R(Q)}$ corresponding to an empty equation in
$Q_1=1$; \item [2)] the number of abelian factors;
\item [3)] the sum of dimensions of surface group factors (that are not embedded into $\Gamma$, \item
[4)] the number of free variables of quadratic equations with
coefficients in $Q_1(X_1,\ldots ,X_n)=1$ corresponding to the fundamental sequence
$Var _{\rm fund}(Q)$, \item [5)] $factors (Var_{\rm fund})$.\end{enumerate} Because $c$ has the same
Kurosh rank, the free factor $F_1$ is unchanged. By 1) and 2)  in
Section \ref{rk3}, abelian and surface factors are sent into
different free factors.

Let $Q_{1i}=1$ be one of the standard quadratic equations in the
system $Q_1=1.$ If the set of solutions of $Q_{1i}=1$ over $F_{R(
Q_2,\ldots Q_n)}$ that factor through the system $S=1$ is a generic family for
$Q_{1i}=1$, then by the analog of (\cite{Imp},  Theorem 9)  we conclude that $S=1$
can be reconstructed so that it contains only one quadratic equation
as a part of the system $S_m=1$. Indeed, suppose a QH subgroup
$\overline Q_{1i}$  corresponding to $Q_{1i}=1$ mapped on some level
$s$ of $S=1$ onto a subgroup of the same size. Then it is stable.
Suppose also that a QH subgroup of $F_{R(Q)}$ that is a subgroup of
$\overline Q_{1i}$ is projected on some level $k$ above $s$ into a
QH subgroup $\overline Q_k$. Then this projection is a monomorphism.
On all the levels above $s$ we can adjoin the image of a subgroup of
$\overline Q_{1i}$ to a non-QH subgroup adjacent to it (and not
count it in the size). We can adjoin the image of $\overline Q_{1i}$
to a non-QH subgroup on all the levels above $m$, and replace the
image of it on the level $m$ by the isomorphic copy of $\overline
Q_{1i}$.

If all QH subgroups corresponding to  $Q_1=1$ are stable, then the
regular size of $S=1$ is the same as the regular size of $Q_1=1$ and
if $ab(c)=ab(Var_{\rm fund}(Q))$, then reconstructed $S=1$ has only
one level.

The lemma  is proved.
\end{proof}

To finish the proof of  Theorem \ref{AE7}, notice that by Lemma \ref{ch},
every time we apply the transformation of Case 3 (we refer to the
cases from Section \ref{nth}) in the construction of $T_{EA}(G)$  we
either (i) decrease the dimension in the top block, therefore decrease the Kurosh rank, or (ii) replace
the NTQ system in the top block  by another NTQ system of the same
dimension but of a smaller size, or (iii) decrease $ab(c)$. Hence the complexity defined in the beginning of this section decreases. Hence,
Case 3 cannot be applied infinitely many times to the top block. If
we apply Case 2, we consider the second block for  proper quotients
of a finite number of groups. Hence, starting from some step, we
come to a situation, when the fundamental sequences factor through
the same block-NTQ system, and the image $G_t$ of $G$ in the last level of
these systems  is a proper quotient of $G$.  
Case 1 cannot appear infinitely many times because every time the induced fundamental sequence has the the same decomposition on levels above some level $k$ and has a decrease in the complexity
of the decomposition on level $k$. Theorem \ref{AE7} is proved. $\Box$

The effectiveness of the  construction of the finite $\forall\exists$-tree for  sentence
$\Phi$ implies that the $\forall\exists$-theory of the group $\Gamma$ is decidable.
This proves Theorem \ref{aetheory}.

\newpage
\section{Effectiveness of the global bound in finiteness results}
 In this section we will give a proof of the effectiveness of the global bound in Theorem 11 \cite{KMel} and show how to generalize the proof for the hyperbolic group case.
In Section 5.4 of \cite{KMel} we defined the notion of a sufficient splitting of a $\Gamma$-limit group $K$ modulo a class of subgroups ${\mathcal K}$.
Let $\Gamma$ be a non-elementary torsion-free hyperbolic group with generators $A$, $P=A\cup \{p_1,\ldots ,p_k\}$, $H=\langle P\rangle .$ Let ${\mathcal K}$ consist of one subgroup ${\mathcal K}=\{H\}.$ Suppose that $K$  does not have a sufficient splitting modulo $H$.  Consider an one-level NTQ system corresponding to the abelian JSJ decomposition of $K$ modulo $H$ (if such a decomposition exists). Denote by $D$ the decomposition of the corresponding NTQ group $N$ modulo $H$. We need to add letters  (extending centralizers) to the generating set of $K$, to obtain a generating set of $N$. Let $N$ be given as the coordinate group of  a finite system of equations $S (X,P)=1.$

 Let $K_1$  be a fully residually $\Gamma$ quotient of the group $K$, $\kappa: K \rightarrow
K_1$ the canonical $\Gamma$- epimorphism that embeds terminal subgroups of $\Gamma$ participating in the construction of $K$, and $H_1 = H^\kappa$ the canonical
image of $H$ in $K_1$.
An elementary abelian splitting of $K_1$ modulo $H_1$ which does not
lift into  $K$ is called a {\em new} splitting.
\begin{definition}
\label{de:reducing} (Definition 20 \cite{KMel}) In the notation above the quotient $K_1$ is
called {\em reducing} if one of the following holds:
\begin{enumerate}
\item $K_1$ has  a non-trivial free decomposition modulo $H_1$;
\item $K_1$ has a new elementary abelian splitting modulo $H_1$.

\end{enumerate}
\end{definition}

 We
say that a homomorphism $\phi :K\rightarrow K_1$ is {\em special}
if $\phi $ either maps an edge group of $D$ to the identity  or
maps
 a non-abelian vertex group of $D$ to an abelian subgroup.
  All  homomorphisms that we consider are $\Gamma$-homomorphisms, therefore they embed the conjugates of subgroups of $\Gamma$ into conjugates of $\Gamma$. 

We will now define $\sim_{MAX}$-equivalent homomorphisms (that were introduced  in \cite{KMel}, Section 5.3). Let $S$ be an elementary abelian splitting of a fully residually $\Gamma$ group $G$  relative to a family of subgroups ${\mathcal
K}$, i.e., $G=A*_{C}B$  or $G=A*_{C}=\langle A,t|c^{t}=c', c\in
C\rangle$ modulo $\mathcal K$. Suppose, for certainty, that $\Gamma \leq
A$.
 Let $\psi:G\rightarrow \Gamma$  be an $\Gamma$-homomorphism from $G$
into $\Gamma$ and  $C^{\psi} \leq \langle c_0\rangle$, where $\langle
c_0\rangle$  is a maximal abelian subgroup of  $\Gamma$. For an arbitrary
$d \in \langle c_0\rangle$ we define a homomorphism $\psi_{d}
:G\rightarrow $ as follows. If $G=A*_{C}B$ then
 $$\psi_{d}(a)=\psi (a) \ for \ a \in A, \ \ \ \psi_{d}(b)=\psi (b)^{d} \ for \  b\in B.$$
   If
$G =\langle A,t|c^{t}=c', c\in C\rangle$ then
 $$\psi_{d} (a)=\psi (a) \  for  \ a\in A, \ \ \ \psi_{d}(t)=d \psi(t).$$
By $\sim_S$ we denote the following binary relation on $Hom_\Gamma(G,\Gamma)$
(in the notation above)
 $$ \sim_S = \{(\psi,\psi_d) \mid \psi \in Hom_\Gamma(G,\Gamma), d \in \langle
 c_0\rangle \}.$$

Now let $D$ be an   abelian JSJ decomposition of $G$ modulo
${\mathcal K}$. Suppose  $M$ is an abelian vertex group in $D$. Then
$M$ is a direct product $M=M_1\times M_2,$ where $M_1$ is the
minimal direct summand of $M$ containing all the edge groups of $M$
in $D$ (so the subgroup generated by the edge groups of $M$ has a
finite index in $M_1$). Denote by $G'$ the subgroup of $G$ which is
the fundamental group of the splitting $D'$ obtained from $D$ by
removing the direct summand $M_2$ from the vertex $M$. Clearly, $G$
splits as an extension of centralizer $C_{G'}(M_1)$ of the group
$G'$ by $M_2$. We fix a basis $g_1, \ldots, g_s$ of the free abelian
group $M_2$ (if $M_2 \neq 1$).

Now let $\theta : G \rightarrow \Gamma$ be an $\Gamma$-homomorphism and
$M^{\theta} \leq \langle c_0\rangle$, where $\langle c_0\rangle$ is
a maximal abelian subgroup of  $\Gamma$. Then for every tuple $d = (d_1,
\ldots, d_s) \in \langle c_0\rangle^s$ the map
  $$\theta_d : g_i \rightarrow d_ig_i^\theta, \ \ i = 1, \ldots,
  s$$
 extends to a homomorphism $\theta_d : M_2 \rightarrow F$.
 Now the restriction of the homomorphism $\theta$ on $G'$ and the
 homomorphism $\theta_d : M_2 \rightarrow \Gamma$ give rise to a
 homomorphism $G \rightarrow \Gamma$ which we define by the same symbol
 $\theta_d$. We refer  to the homomorphism $\psi_{d}$ and $\theta_d$
as obtained from $\psi$ and $\theta$ by {\em extended
automorphisms} or {\em fractional Dehn twists}.

By $\sim_M$ we denote the following binary relation on $Hom_\Gamma(G,\Gamma)$
(in the notation above)
 $$ \sim_M = \{(\theta,\theta_d) \mid \theta \in Hom_\Gamma(G,\Gamma), d \in \langle
 c_0\rangle^s \}.$$

We extend the relation $\sim$ of being equivalent with respect to the group of canonical automorphisms to the equivalence relation $\sim
_{AE}$ generated by $\sim$, all the binary relations $\sim_M$ where
$M$ runs over all abelian vertex groups in $D$, and all the binary
relations $\sim_S$ where $S$ runs over all elementary splittings of
$G$ corresponding to the edges of $D$.

We say that two $\Gamma$-homomorphisms $\phi, \psi \in Hom_\Gamma(G,\Gamma)$
are ${MAX}$-equivalent (and write $\phi \sim_{MAX} \psi$) if
 there exists $\theta \in Hom_\Gamma(G,\Gamma)$ such that
  $\phi \sim _{AE} \theta$ and $\theta$  coincides up to conjugation with $\psi$
   on the fundamental group of
  every connected component of the graph of groups obtained from
  $D$ by   removing from $D$ all QH-subgroups. 

Let ${\mathcal R}=\{K/R(r_1), \ldots, K/R(r_s)\}$ be a complete
reducing system for $K$ (each homomorphism from $K$ into $\Gamma$ that factors through a reducing quotient is $\sim_{MAX}$-equivalent  to a homomorphism that factors through one of them).  The existence of such system  for a free group is proved in \cite{KMel}, this can be similarly proved for a torsion free hyperbolic group.  Suppose we know NTQ groups for the system of reducing quotients of $K$ modulo $H$.  A homomorphism from $K$ onto $\Gamma$ is called   {\em
reducing}
  if there exists a solution   the $\sim_{MAX}$-equivalence class
  of $\psi$ which  factors through  one of the NTQ systems for equations $r_1= 1, \ldots, r_k = 1.$
Now we define algebraic  solutions of $S = 1$ in $\Gamma$.
  Let $\phi :H \rightarrow \Gamma$ be a
fixed $\Gamma$-homomorphism and $Sol_\phi$ the set of all homomorphisms
from $K$ onto $\Gamma$ which extend $\phi$.
A non-reducing non-special solution in $Sol_\phi$ is called
{\em $K$-algebraic} (modulo $H$ and $\phi$).

\begin{theorem} (cf. \cite{KMel1}, Theorem 6) \label{th5} Let $H\leq K$ be as above. The fact that for parameters $P$ there are exactly $n$ non-equivalent Max-classes of $K$-algebraic solutions
of the equation $S(X,P)=1$ modulo $H$ can be written algorithmically as a boolean combination of conjunctive $\exists\forall$-formulas, namely formulas of type \begin{equation}\label{AE} \forall X\exists Y (U(X,P)=1\implies V(X,P,Y)=1).\end{equation}

\end{theorem}
\begin{proof} The generating set
$X\cup P$ of $N$ corresponding to the decomposition $D$ can be partitioned
 as $X=X_1\cup X_2\cup P$ so that $G=\langle X_2\cup P\rangle $ is the  fundamental group of the graph of groups obtained from $D$ by removing all QH-subgroups.
If $c_e$ is a given generator of an edge group of $D$, then we know how a generalized fractional Dehn twist (AE-transformation or extended automorphism in the terminology of \cite{KMel}, \cite{Imp}) $\sigma$ associated with edge $e$ acts on the generators from the set $X$. Namely, if $x\in X$ is a generator of a vertex group, then either $x^{\sigma}=x$, or $x^{\sigma}=c^{-m}xc^{m},$ where $c$ is a root of the image of $c_e$ in $\Gamma$, or, in case when $e$ is an edge between abelian and rigid vertex groups and $x$ belongs to the abelian vertex group, $x^{\sigma}=xc^{m}$. Similarly, if $x$ is a stable letter then either $x^{\sigma}=x$, or $x^{\sigma}=xc^{m}.$

 One can write elements $c_e$ as words in generators $X_2$, $c_e=c_e(X_2)$, because all edge groups belong to $G$. Denote $T=\{t_i,\ i=1,\ldots ,m\}.$ Consider the  formula
\begin{multline*}
\exists X_1 \exists X_2\forall Y \forall T\forall Z \left (
S(X_1\cup X_2,P)=1\right.\\  \wedge \neg \left(\left.\bigwedge_{i=
1}^{m}[t_i,c_i(X_2)]=1 \wedge Z=X_2^{\sigma_T}\wedge
S(Y\cup X_2,P)=1\wedge V(Y\cup Z,P)=1\right)\right ).
\end{multline*}
It says that there exists a solution of the equation $S(X_1,X_2,P)=1$ that is not Max-equivalent to a solution $Y,Z,P$ that
satisfies $V(Y,Z,P)=1$.  If now $V(Y,Z,P)=1$ is a disjunction of equations defining maximal reducing quotients (in the case we know them)  or NTQ systems for maximal reducing quotients (see Proposition \ref{cred} below), then this formula states that for
parameters $P$ there exists at least one Max-class of algebraic solutions of $S(X,P)=1$ with respect to $H$.

Denote $$\tau(T,X_2,Y,Z)=\left(\bigwedge_{i=
1}^{m}[t_i,c_i(X_2)]=1 \wedge Z=X_2^{\sigma_T}\wedge
S(Y\cup X_2,P)=1\wedge V(Y\cup Z,P)=1\right).$$The following formula states  that for
parameters $P$ there exist at least two non-equivalent Max-classes of algebraic solutions of $S(X,P)=1$ with respect to $H$.
\begin{multline*}
\theta _2(P)=\exists X_1, X_3 \exists X_2, X_4\forall Y, Y' \forall T, T',T''\forall Z,Z' \left(
S(X_1,X_2,P)=1\wedge S(X_3,X_4,P)=1\right.\\  \wedge \neg \left(\left.\tau(T,X_2,Y,Z)\vee \tau(T',X_4,Y',Z')\vee (\bigwedge_{i
=1}^{m}[{t_i}'',c_i(X_2)]=1\wedge X_2^{\sigma_{T''}}=X_4) \right)\right).
\end{multline*}

 Similarly one can write a formula $\theta _n(P)$ that states for
parameters $P$ there exist at least N non-equivalent Max-classes of algebraic solutions of $S(X,P)=1$ with respect to $H$.

Then $\theta _n(P)\wedge \neg\theta _{n+1}(P)$ states that there are exactly $n$ non-equivalent Max-classes. The theorem is proved.
\end{proof}

\begin{theorem}\label{term} (for the case when $\Gamma$ is a free group this is \cite{KMel}, Theorem 11) Let $H, K$ be finitely generated  fully residually $\Gamma$
groups such that $ \Gamma\leq H\leq K$ and $K$ does not have a
sufficient splitting modulo $H$. Let  $D$ be an abelian JSJ
decomposition of $K$ modulo $H$ (which may be trivial).  There
exists a constant $n=n(K,H)$
 such that for each $\Gamma$-homomorphism
 $\phi:H \rightarrow \Gamma$ there are at most $n$
algebraic  pair-wise non-equivalent with respect to  $\sim _{MAX}$, homomorphisms from
$K$ to $\Gamma$ that extend $\phi$.

 Moreover, if $H,K$ are as in Theorem \ref{th5}, the constant $n$ for the number of $\sim
 _{MAX}$-non-equivalent homomorphisms
 can be found effectively.
\end{theorem}
\begin{proof}  The statement about the existence of such a constant $n$ is Theorem 3.5 \cite{Sela} (although there is no proof of Theorem 3.5 there).
We will show how to find this constant effectively.
To make presentation easier, we consider first the case when the group $K$ from the formulation of the theorem does not have a splitting modulo $H$.
(in the terminology of \cite{Sela} it is a {\em rigid} limit group). We consider the formula
$$\exists P\exists Y_1,\ldots ,Y_m (\wedge _{i=1}^m S(P,Y_i)=1\wedge Y_i\neq Y_j (i\neq j)\wedge _{t=1}^k\wedge _{i=1}^m r_t(P,Y_i)\neq 1).$$
We know from  \cite{Sela}, Theorem 3.5, that the number $m$ of possible algebraic solutions is bounded. Therefore for some  positive integer $m$ such a formula will be false. The minimal such $m$ can be found because the existential theory of $\Gamma$ is decidable. Therefore $n=m-1$.

Now we consider the case when the group $K$ has a splitting modulo $H$ but not a sufficient splitting ($K$ is {\em solid} in terminology of \cite{Sela}). This case is more complicated because we have to write  that  solutions corresponding to tuples  $Y_i, Y_j, i\neq j,$  are not reducing and belong to different $\sim _{MAX}$-equivalence classes. This means that there exist no elements representing QH subgroups and no elements commuting with edge groups of the JSJ decomposition of $K$ modulo $H$ such that application of generalized fractional Dehn twists corresponding to these elements  take some of these solutions to reducing solutions or take one solution to the other.   This fact can be  expressed in terms of an $\exists\forall$-sentence that is true if and only if there exists a homomorphism $H\rightarrow F$ that can be extended to $m$ algebraic and not $\sim _{MAX}$ equivalent homomorphisms $K\rightarrow F.$ The decidability of the $\exists\forall$-theory of $\Gamma$ was proved in the previous section. Then the bound on $m$ can be found effectively because we can  find out for which $m$ the sentence is false and therefore such a homomorphism $H\rightarrow F$ does not exist.
\end{proof}

\newpage
\section{Quantifier elimination algorithm}\label{sec:6}
In this section we will prove Theorem \ref{main}. Consider the following formula

\begin{equation}\label{38neg} \Theta(P)=\exists Z \forall X\exists Y (U(A,P,Z,X,Y)=1\wedge V(A,P,Z,X,Y)\not =1),\end{equation}
where $A$ is a generating set of $\Gamma$.
This formula $\Theta(P)$ is the negation of the formula $\Phi$ considered in \cite{KMel}.

The existence of quantifier elimination to boolean combinations of $\forall\exists$-formulas for $\Gamma$ was proved in \cite{Sela}. Earlier it was proved in \cite{Sela5},\cite{KMel} that every formula in the theory of a free group $F$ is equivalent to a boolean combination of $\forall\exists$-formulas. The general schemes of the proofs in \cite{Sela5} and in \cite{KMel} is quite similar: to use the  implicit function theorem, which is Theorem \ref{IFT} for a torsion free hyperbolic group (= existence of formal solutions
in the covering closure of a limit group) and to approximate any definable set and get its stratification  using certain verification process (based on the implicit function theorem) that stops after a finite number of steps. But all the necessary technical results are proved differently (using actions on $\mathbb R$-trees in \cite{Sela5}, and using elimination process and free actions of fully residually free groups on ${\mathbb Z}^n$-trees, which is equivalent to the existence of free length functions in ${\mathbb Z}^n$, in \cite{KMel}). The proof in \cite{KMel} is also algorithmic. It will be more convenient for us to follow our proof in \cite{KMel} and to use our terminology but refer to \cite{Sela} for necessary technical results.

To obtain effective quantifier elimination to boolean combinations of $\forall\exists$-formulas it is enough to give an algorithm to find such a boolean combination that defines the same set  as  $\Theta(P).$

The procedure for a torsion free hyperbolic group $\Gamma$ is similar to the one for a free group. We recall how the procedure goes. For every tuple of elements $\bar P$ for which $\Theta(\bar P)$ is true, there exists some $\bar Z$ and (by the Merzljakov theorem  (Theorem  2.1, \cite{Sela})) a solution $Y=f(A,\bar P,\bar Z,X)$ of $U=1\wedge V\neq 1$ in $F(X)\ast \Gamma.$ All formula solutions of $U=1$ for all possible values of $P$ belong to a finite number of fundamental sequences with terminal groups $\Gamma _{R(U_{1,i})}\ast F(X),$ where $U_{1,i}=U_{1,i}(A,P,Z,Z^{(1)})$ and $\Gamma _{R(U_{1,i})}$ is a group with no sufficient splitting modulo $\langle A,P,Z\rangle$ (this is done entirely similar to the case of a free group which described in Section 12.2, \cite{KMel}). These groups can be  found by Proposition \ref{12} as quasi-convex closures of terminal groups of constructed fundamental sequences. 
 
We now consider each of these fundamental sequences separately.  Below we will not write the constants $A$ in the equations but assume that equations may contain constants. Those values $P,Z$ for which there exists a value of $X$ such that the equation 
$$V(P,Z,X,f(Z, Z^{(1)},P, X))=1$$ is satisfied for any function $f$ give a system of equations on $\Gamma _{R(U_{1,i})}\ast F(X).$ This system is equivalent to a finite subsystem (to one equation in the case when we consider formulas with constants).  Let $G$ be the coordinate group of this system and  $G_i, i\in J$ be the corresponding fully residually $\Gamma$ groups.

We introduced in Section 12.2, \cite{KMel}, the tree $T_{X}(G)$ which is constructed  (in the case of a free group $F$) the same way as $T_{EA}(G)$ with $X, Y$ considered as variables and $P,Z, Z^{(1)}$ as parameters.  Entirely similar such a tree can be constructed in the case of $\Gamma$ as follows. To each group $G_i$  we assign fundamental sequences modulo $\langle P,Z, Z^{(1)}\rangle$. Their terminal groups are groups  $\Gamma _{R(V_{2,i})}$, where
  $$V_{2,i}=V_{2,i}(P,Z,Z^{(1)}, Z_1^{(2)})$$
  that do not have a sufficient splitting modulo $\langle P,Z, Z^{(1)}\rangle$. Then we find all formula solutions $Y$ of the equation $$U(P,Z,X,Y)=1$$ in the corrective normalizing extensions of the NTQ groups corresponding to these fundamental sequences for $X$ (see  \cite{Imp}, Theorem 12). These formula solutions $Y$ are described by a finite number of fundamental sequences with terminal groups $F_{R(U_{2,i})},$ where $U_{2,i}=U_{2,i}(P,Z,Z^{(1)},Z_{1}^{(2)}, Z^{(2)}).$ Then again we investigate the values of $X$ that make the word $V(P,Z,X,Y)$ equal to the identity for all these formula solutions $Y$. And we continue the construction of $T_{X}(G).$ We can prove that this tree is finite exactly the same way as we proved the finiteness of the $\forall\exists$-tree.  We will call $T_{X}(G)$ {\em the  parametric $\forall\exists$-tree} for the formula $\Theta(P)$. For each branch of the  tree $T_{X}$ we assign a
 sequence of toral relatively hyperbolic $\Gamma$-limit groups $$\Gamma _{R(U_{1,i})},
\Gamma _{R(V_{2,i})}\ldots ,\Gamma _{R(V_{r,i})}, \Gamma _{R(U_{r,i})}$$
as in \cite{KMel}, Section 12.2. Corresponding irreducible systems of equations are:

$$U_{1,i}=U_{1,i}(P,Z,Z^{(1)}),$$
$$U_{m,i}=U_{m,i}(P,Z,Z^{(1)},
Z_1^{(m)}, Z^{(m)}),\ m=2,\ldots ,r,$$ which correspond to the terminal
groups of fundamental sequences describing $Y$  of level
$(m,m-1)$, and
 $$V_{m,i}=V_{m,i}(P,Z,Z^{(1)}, Z_1^{(m)}),\ m=2,\ldots ,r$$
which correspond to the terminal groups of fundamental sequences describing $X$
 of level $(m,m)$. They correspond to vertices of $T_{X}$
that have distance $m$ to the root.

For each $m$ the group $\Gamma _{R(U_{m,i})}$ does not have a sufficient
splitting modulo the subgroup $\langle P,Z,Z^{(1)}, Z_1^{(m)}\rangle ,$ and the group
$\Gamma _{R(V_{m,i})}$ does not have a sufficient splitting modulo the
subgroup  $\langle P,Z,Z^{(1)}\rangle .$

On each step we consider terminal groups of
all levels. Below we will sometimes skip index $i$ and write $U_m,\
V_m$ instead of $U_{m,i},\ V_{m,i}.$

\begin{prop}  \label{cred}  Let $S(Z, A)=1$ be a finite system of equations over $\Gamma$, and $H\leq \Gamma _{R(S)}$. Let  $K$ be  a terminal group of a completed fundamental sequence modulo $H$, and
$K$  does not have a sufficient splitting modulo $H$. Then there is an algorithm to construct a complete system of corrective extensions of completed canonical fundamental sequences modulo $H$ for a complete system of reducing quotients $K_1,\ldots ,K_m$  of $K$.\end{prop}

\begin{proof}
As in the proof of Proposition \ref{12},  using canonical representatives we can construct a family of generalized equations $\Omega _1,\ldots , \Omega _k$ in the free group $F$   ($\pi :F\rightarrow \Gamma$) \cite{Imp} such that each solution of each $\Omega _i$ in $F$ (as a system of equations in the group) corresponds to  a solution of $S(Z,A)=1$ in $\Gamma$, and every solution of $S(Z,A)=1$ in $\Gamma$ corresponds to some solution of some $\Omega _i$ as a generalized equation (solution without cancellations).  We can run the Elimination process for each generalized equation modulo the pre-image of $H$ (generators of $H$ are included in the set of variables of $S(Z,A)=1$).  If a generalized equation corresponds to a generic family of solutions for a freely indecomposable  NTQ system or a freely indecomposable NTQ system modulo a subgroup, then all the splittings on all the levels of this NTQ system are detected in the Elimination process \cite{KMS} (we called it Makanin's process in \cite{Imp}) for the generalized equation, and  produce a corresponding NTQ system over a free group.  Moreover, the edge groups for these splittings are not trivialized in the re-working process described in Section \ref{section:HomDiagrams}. We will obtain in the Elimination process completed  fundamental sequences ending with groups ${K_i^*}$ without sufficient splitting modulo the group $H^*$ generated by $F$ and the variables corresponding to the generators of $H$. Then we will obtain completed fundamental sequences for reducing quotients of  groups  ${K_i^*}$.   Modifying the obtained by the Elimination process  NTQ groups for reducing quotients ${K_i^*}$ over $F$ into NTQ groups over $\Gamma$ for quotients of $K$, we will obtain NTQ groups for different quotients of $K$ but, in particular, we will obtain them for all the maximal reducing quotients of $K$ because the new splittings for quotients of $K$ will be seen in the Elimination process over a free group. Then we  have to compare the  reducing quotients that we obtain and take the Sol-maximal ones.  The procedure for finding Sol-maximal $\Gamma$-limit quotients in described in  \cite{KMT}, Section 6.1.
\end{proof}
\subsection{Algorithm for the construction of the tree $T_X(G)$.}
\begin{prop}\label{TX}

 There is an algorithm to construct the following: 
 \begin{enumerate}\item [1)] the finite parametric $\exists\forall$-tree $T_{X}(G) $, \item [2)] for each branch of the tree $T_{X}$  the finite family of toral relatively hyperbolic $\Gamma$-limit groups $$\Gamma _{R(U_{1,i})},
\Gamma _{R(V_{2,i})}\ldots ,\Gamma _{R(V_{r,i})}, \Gamma _{R(U_{r,i})}$$ described above. Each group is a quasi-convex closure of a group described as a  maximal $\Gamma$-limit quotient of a group given by a finite system of equations over $\Gamma$.   \item [3)]  for each vertex of the tree, a fundamental sequence describing either $Y$ (if the associated group is $\Gamma _{R(U_{j,i})}$) or $X$ (if the associated group is $\Gamma _{R(V_{j,i})}$).\end{enumerate}

\end{prop}
\begin{proof}

 The proof follows Section 12.2, \cite{KMel} (but works with $\Gamma$ instead of $F$) and  uses the algorithm from Proposition \ref{cred} to construct a complete system of NTQ groups (and fundamental sequences) for reducing quotients and Proposition \ref{12} that 
states that we can construct fundamental sequences modulo a finite set of finitely generated subgroups algorithmically.  Indeed,
a fundamental sequence describing  $Y$ on level $(m,m-1)$ terminates at the group  $\Gamma _{R(U_{m,i})}$   and a fundamental sequence describing $X$ on level $(m,m)$ terminates at the  group  $\Gamma _{R(V_{m,i})}$. \end{proof}

The tree $T_{X}(G)$ is finite, as in \cite{KMel} we  have schemes of levels $(1,0),(1,1),(2,1),(2,2)$ etc up to some number $(m,m)$.

\subsection{Configuration groups}

We will concentrate on level $(2,1)$ now. In Definition 27 and Definition 28, \cite{KMel} we define initial fundamental sequences of levels $(2,1)$ and $(2,2)$  and width $i$ (the possible width is bounded) modulo $P$.  Since we are now considering the formula $\Theta$ such that $\Theta=\neg\Phi$ for the formula $\Phi$ considered in \cite{KMel}, we will slightly change the definition here. It will be more convenient to replace condition (6) from Definition 28 of \cite{KMel} by its negation and add this negation on level $(2,1)$.

\begin{definition}  \label{d2} Let $\Gamma _{R(V_{2,1})},\ldots ,\Gamma _{R(V_{2,t})}$ be the whole family
of groups on level $(1,1)$ constructed for a fixed group $\Gamma _{R(U_{1,k})}$ ($k$ is fixed).  To construct the {\em initial
fundamental sequences of level} (2,1) and {\em width} $i=i_1+\ldots
+i_t$, we consider the fundamental sequences modulo the subgroup
$\langle P\rangle$ for the groups $H$ discriminated
by $i$ solutions of the systems
$$U_{2,m_s}(P,Z,Z^{(1)},Z_1^{(2,j,s)},Z^{(2,j,s)})=1,\ j=1,\ldots i_s,\ s=1,\ldots ,t,$$ with the properties:

(1) $Z^{(1)}$ are algebraic solutions of $U_{1,k} (P,Z,Z^{(1)})=1$, $Z_1^{(2,j,s)}$ are algebraic solutions of $V_{2,s} (P,Z,Z^{(1)}, Z_1^{(2)})=1$,   $Z^{(2,j,s)}$ are algebraic solutions of $U_{2,m_s} (P,Z,Z^{(1)}, Z_1^{(2)}, Z^{(2)})=1;$

(2) $ Z_{1}^{(2,j,s)}$ are not MAX-equivalent to $ Z_{1}^{(2,p,s)},
p\neq j,\ p,j=1,\ldots, i_s,\ s=1,\ldots ,t$;

(3) for any of the finite number of values of $Z_{1}^{(2)}$
the fundamental sequences for $V_{2,s}(P
,Z,Z^{(1)},Z_{1}^{(2)})=1$ are contained in the union
of the fundamental sequences for $U_{2, m_s}(P
,Z,Z^{(1)},Z_{1}^{(2,j)},Z^{(2,j)})=1$ for
different values of $Z^{(2,j,s)}$;

(4) there is no non-equivalent $Z_{1}^{(2,i_{s}+1,s)}$, algebraic,
solving $V_{2,s}(P,Z,Z^{(1)}, Z_1^{(2)})=1, s=1,\ldots ,t$.

(5) the solution $P,Z,Z^{(1)}$ does not
satisfy a proper equation which implies $V=1$ for any value of
$X$.

 (6)  for any $s$,  the solution $P
,Z,Z^{(1)},Z_{1}^{(2,1,s)},Z^{(2,1,s)}$  can not be
extended to a solution of some
$$V_{3,s}(P,Z,Z^{(1)},Z_{1}^{(2,1,s)},Z^{(2,1,s)},
Z_{1}^{(3,1,s)})=1.$$

We call this group $H$ a {\em configuration group}. We also call a tuple $$Z,Z^{(1)},Z_{1}^{(2,j,s)},Z^{(2,j,s)},\ j=1,\ldots i_s,\ s=1,\ldots ,t$$ satisfying the conditions above {\em a certificate} for $\Theta$ for $P$ (of level (2,1) and width $i$). 
 We add to the generators of the configuration group additional variables $Q$ for the primitive roots of a fixed set of elements for each certificate (these are primitive roots of the images in $Gamma$ of the edge groups and abelian vertex groups in the relative JSJ decompositions of the groups $\Gamma _{R(V_{2,1})}$). \end{definition}
 
 Each group $H$ from this definition is a fundamental group of some system of equations, say $$W(P,Z,Z^{(1)},Z_1^{(2,j,s)},Z^{(2,j,s)}, Q, \ j=1,\ldots i_s,\ s=1,\ldots ,t)=1.$$ For each initial
fundamental sequence of level (2,1) and width $i$, Identically to the proof of Lemma 27 \cite{KMel}, one can show that 
 for each value of parameters $P$ factoring through this fundamental sequence for which there exists a certificate, there are the following possibilities: 
 \begin{enumerate}\item  there exists 
 a generic family of certificates (corresponding to the fundamental sequence), 
 \item any certificate in this fundamental sequence can be extended by $Z_1^{(2,i_s+1,s)}$ so that the whole tuple factors through one of the groups $H_{surplus}$ discriminated by solutions of $W=1$ together with solutions $Z_1^{(2,i_s+1,s)}\rightarrow\Gamma$  minimal with respect to fractional Dehn twists,  and going through one of the fundamental sequences for which $Z_1^{(2,i_s+1,s)}$  is either reducing or Max-equivalent to one of $Z_1^{(2,j,s)}, j=1, \ldots ,i_s.$\end{enumerate}
In the former case we say that the fundamental sequence has depth 1, it the latter case
we will consider fundamental sequences of depth 2 (for level $(2,1)$ and width $i$).

 Notice that we do not know a system of equations defining  a configuration group. We, therefore, need the following result.

\begin{prop} Let $H=\Gamma _{R(W)}$ be one of the configuration groups with generators
$$P, Z, Z^{(1)},
Z_{1}^{(2,j,s)},Z^{(2,j,s)}, Q, \  j=1,\ldots ,i_s,\ s=1,\ldots
,t. $$    Then there is an algorithm to find  a quasi-convex closure of each terminal group of each fundamental sequence  for $H$ modulo $P$.
\end{prop}

\begin{proof} We will first prove the statement of the proposition for the case when $\Gamma $ is a free group.  Let $\Gamma =F$. As in the proof in  of Theorem 11\cite{KMel}, we extensively use the technique of {\em generalized
equations}  described in \cite{Imp}, Subsection 4.3 and Section 5 and {\em cut equations} described in Section 5.7 \cite{Imp}.
The reader has to be familiar with these sections of \cite{Imp}. In the proof of Theorem 11, \cite{KMel} we show how to construct, given a group $K$ that does not have a sufficient splitting modulo a subgroup $H$, a finite system of cut equations $\Pi$ (see \cite{Imp}, Section 7.7) for a minimal in its Max-class solution such that the intervals of $\Pi$ are labeled by values of the generators of $H$. For each system $$U_{2,m_s}(P,Z,Z^{(1)},Z_{1}^{(2,j,s)},Z^{(2,j,s)})=1$$ we construct a cut equation modulo the parametric subgroup $\langle P,Z,Z^{(1)},Z_{1}^{(2,j,s)}\rangle $. The intervals of this cut equation are labeled by $P,Z,Z^{(1)},Z_{1}^{(2,j,s)}$.
For the intervals labeled by $P,Z,Z^{(1)}$ we add a cut equation for the system
 $$V_{2,m}(P,Z,Z^{(1)}, Z_{1}^{(2,j,s)})=1$$ modulo the parametric subgroup $\langle P,Z,Z^{(1)}\rangle .$
 For the intervals labeled by $P,Z$ we add  cut equations for the system
 $$U_{1,m}(P,Z,Z^{(1)})=1$$ modulo the parametric subgroup $\langle P,Z\rangle .$ The intervals labeled by $P$ will be the same for all these cut equations. Similarly we identify all the intervals labeled by the same variables that occur in different cut equations.
We now add to the obtained cut equation, which can be also considered as a generalized equation, the inequalities that guarantee that conditions (1)-(6) are satisfied. These inequalities are just indicating that
specializations of variables corresponding to some sub-intervals of the generalized equation must not be identities. But this is a standard requirement  for a solution of a generalized equation. For example, we write an equation $r_1(Z_{1}^{(2,j)})=\lambda _1$ and set that $\lambda _1$ is a base of the generalized equation. Then the condition $\lambda _1\neq 1$ must be automatically satisfied for a solution of a generalized equation.  So we can construct a finite number of generalized equations such that each certificate corresponding to minimal in their Max-classes specializations is a solution of one of these generalized equations $\mathcal {GE}$.

We now construct fundamental sequences of solutions of the equations $\mathcal {GE}$ modulo $\langle P\rangle $.  Notice that not all solutions from the fundamental sequence
satisfy the necessary inequalities, but if we restrict the sets of automorphisms on all the levels to those whose application preserves corresponding generalized equations, we will have solutions of inequalities too. Therefore, a generic family of solutions does satisfy the inequalities. Using our standard procedure we construct fundamental sequences induced by the subgroup with generators $$P, Z, Z^{(1)},
Z_{1}^{(2,j,s)},Z^{(2,j,s)}, Q,\  j=1,\ldots ,i_s,\ s=1,\ldots
,t.$$  The subgroups generated by the images of these generators in the terminal groups of these fundamental sequences are precisely the terminal groups of the fundamental sequences for $H$ modulo $P$.

Let now $\Gamma$ be a torsion free hyperbolic group.  For each system $$U_{2,m_s}(P,Z,Z^{(1)},Z_{1}^{(2,j,s)},Z^{(2,j,s)})=1$$ over $\Gamma$ we construct a system of equations over $F$ using canonical representatives and a cut equation modulo parameters subgroup $\langle P,Z,Z^{(1)},Z_{1}^{(2,j,s)}\rangle $ for this system.   For the intervals labeled by $P,Z$ we add  equations in a free group constructed using canonical representatives for  the system
 $$U_{1,m}(P,Z,Z^{(1)})=1$$ over $\Gamma$ 
and add cut equations modulo the parameters subgroup $\langle P,Z\rangle $. The inequalities over $\Gamma$ will correspond to inequalities over $F$ which, again, indicate that specializations of some variables of the generalized equation must be non-trivial.   We construct fundamental sequences modulo $\langle P\rangle $.  Then we transform these fundamental sequences into fundamental sequences over $\Gamma$  as we  did in the proof of Proposition \ref{prop7}. Then we construct fundamental sequences induced by the subgroup with generators $$P, Z, Z^{(1)},
Z_{1}^{(2,j,s)},Z^{(2,j,s)}, Q,\  j=1,\ldots ,i_s,\ s=1,\ldots
,t.$$ 
 The subgroups generated by the images of these generators in the terminal groups of these fundamental sequences are the groups we are looking for.
\end{proof}

This implies the following result.

\begin{cor} There is an algorithm to construct the initial fundamental sequences for $Z$ of level (2,1) and width $i$ related to $\Gamma _{R(V_2)}.$
\end{cor}

 Lemma 27, \cite{KMel},  states that the set of parameters $P$ for which there exists a fundamental sequence of level $(2,1)$ and width $i$ and a certificate, consists of those $P$ for which  there exists a generic family of certificates ({\em generic certificate}) and those for which all the certificates factor through a proper projective image of this fundamental sequence. Actually, Lemma 27 deals with certificates satisfying only properties (1)-(5), but the proof does not change if we add the property (6) to the definition of a certificate.

For a given value of $P$ the formula $\Theta$
can be proved on level  (2,1) and depth 1
 if and only if  the following
conditions are satisfied.
\begin{enumerate}

\item [(a)] There exist algebraic solutions  for some system of equations
$U_{i,coeff}=1$ corresponding to  (a quasi-convex closure of) the terminal group of a fundamental
sequence $V_{i,\rm fund}$ for a configuration group modulo $P$.
\item [(b)] These solutions do not factor through the fundamental sequences that
describe solutions from $V_{i,\rm fund}$ that do not satisfy one of
the properties (1)-(6). There is a finite number of such fundamental
sequences.

\item  [(c)] These solutions do not factor through (the quasi-convex closure of)
 the terminal groups of fundamental sequences of
level (2,1) and greater depth derived from $V_{i,\rm fund}.$

\item  [(d)] $(P,Z,Z^{(1)})$ cannot be
extended to a solution of  $V=1$ by arbitrary $X$ ($X$ of level
0) and $Y$ of level (1,0).
\end{enumerate}

In this case there is a generic certificate of level $(2,1)$ width $i$ and depth 1.
These conditions can be described by a boolean combination of  $\exists\forall$-formulas of type (\ref{AE}).
Similarly we consider fundamental sequences of level $(2,1)$ width $i$ and depth 2 and deeper fundamental sequences of level $(2,1)$ width $i$. We construct the projective tree (see \cite{KMel}, Section 11) to construct these deeper sequences.

 We now need another algorithmic result that states that  the main technical tool of the procedure of constructing the projective tree,  tight enveloping NTQ groups and fundamental sequences, can be effectively constructed.

\subsection{Tight Enveloping NTQ Groups}
We now have to modify the definition of  a {\em tight enveloping NTQ group and fundamental sequence}. This is
done as follows. We begin with a fundamental sequence satisfying first and second restrictions and the NTQ group for it that we denote $\Gamma _{R(L_1)}.$  We denote the fundamental sequence $c(L_1)$. Let $G$ be a subgroup of $\Gamma _{R(L_1)}$  and $G=\Gamma _{R(U)}$  be  the quasi-convex closure of $G$. Our goal is to construct
a fundamental  sequence $c(U)$ and an NTQ group for $G$ such that the Kurosh rank of the NTQ group for $G$ with respect to $c(L_1)$ is the same as the Kurosh rank of $c(U)$. The fundamental sequence $c(U)$ will have also other important properties that we will need later.

\begin{enumerate}\item [(a)] We take the NTQ group induced by the image  of $G$ in $\Gamma _{R(L_1)}$  from
$\Gamma _{R(L_1)}$ as described in Subsection \ref{inher}. This does not
increase the Kurosh rank, because we
 add only elements from abelian subgroups and conjugating elements that are mapped to the identity on the next lower level. Denote this group by $Ind(\Gamma _{R(U)})$. Then we do the
following.

  \item [(b)] We add from the top to the bottom (considering on level $i+1$ the image of the group
extended on level $i$)  those QH subgroups $Q$ of the  group $\Gamma _{R(L_1)}$ that intersect
  $Ind(\Gamma _{R(U)})$ in a subgroup of finite index (in $Q$) and have less free variables than the
  subgroup in the intersection.
  
  \item [(c)] We add edge groups of abelian subgroups of the
enveloping group that have non-trivial intersection with
$Ind(\Gamma _{R(U)})$ if this does not increase the Kurosh rank.

\item [(d)] The terminal level of $Ind(F_{R(U)})$ admits a free decomposition induced from the free decomposition of the terminal group of the NTQ group $\Gamma _{R(L_1)}$, $M=M_1*\ldots M_s*F$, where $F$ is a free group (possibly trivial) and $M_1,\ldots ,M_s$ are embedded into conjugates of the non-cyclic freely indecomposable factors in the free decomposition of the terminal level of  $\Gamma _{R(L_1)}$. We replace each $M_i$ by the factor that contains it.
\item [(e)] We add to $Ind(F_{R(U)})$ from bottom to top all the QH
subgroups of $\Gamma _{R(L_1)}$ that have non-trivial intersection with
$Ind(\Gamma _{R(U)})$, do not have free variables, the corresponding level
of $Ind(\Gamma _{R(U)}$ intersects non-trivially some of their adjacent
vertex groups, and their addition decreases the Kurosh rank.

\item [(f)] We add from bottom to top all the elements that conjugate different $QH$
subgroups (abelian vertex groups) of $Ind(\Gamma _{R(U)})$ into the same
QH subgroup of ${\mathcal L_1}$ if this decreases the Kurosh rank.

\item [(g)]  We add (from bottom to top) to each non-cyclic  factor $H_i$ in  the free decomposition $H_1\ast\ldots\ast H_t\ast F$ of the image of the constructed fundamental sequence on each level,  the abelian vertex groups and  edge groups on this level of $c(L_1)$ that are intersected non-trivially by $H_i$.
\end{enumerate}

We make these steps (which we call adjustment) iteratively and denote the obtained group by $Adj (\Gamma _{R(U)})$.
We  repeat the adjustment iteratively as many times as possible.

We call the constructed NTQ group the tight
enveloping NTQ group. We will also call the corresponding system
(fundamental sequence) the tight enveloping system (fundamental
sequence).  As a size of  a
QH subgroup $Q$ in the tight enveloping NTQ group we consider the
size of the QH subgroup in the enveloping group containing $Q$ as a
subgroup of a finite index.

 Given a fully residually $\Gamma$ group $G=\Gamma _{R(U)}$, the  NTQ system $W=1$ corresponding to a fundamental sequence for $U=1$ (with the quadratic system $S_1(X_1,\ldots, X_n)=1$ corresponding to the top level and its image), a system of equations ${\mathcal P}=1$  with coefficients in $\Gamma _{R(W)}$ having a solution in some extension of $\Gamma _{R(W)}$ we  construct fundamental sequences (satisfying first and second restrictions) for ${\mathcal P}=1$ modulo the non-cyclic free factors of the second level $\langle X_2,\ldots ,X_n\rangle $ of $\Gamma _{R(W)}.$ Consider one of these fundamental sequences and construct the $NTQ$ group  for it.  Denote it $\Gamma _{R(L_1)}$. Suppose $G$ is embedded into $\Gamma _{R(L_1)}$.  Suppose also that the family of simple closed curves that are mapped to the identity in each QH subgroup in $S_1$ is compatible with such a family for the 
 fundamental sequence $L_1$ (splitting of quadratic equations is compatible).

The Kurosh rank of the tight enveloping fundamental sequence $(TEnv
(S_1; L_1))$ extracted from $L_1$ is less than or equal to the Kurosh rank of ${\mathcal S}_1$
modulo free factors of $\langle X_2,\ldots ,X_n\rangle$  (because we
consider only fundamental sequences compatible with the free
factorization of the subgroup $\langle X_2,\ldots ,X_n\rangle$ 
 and compatible with the splitting of quadratic
equations as discussed in Subsections \ref{spl}, \ref{rk3}). If the Kurosh ranks
are the same, we  reorganize the levels of the enveloping
system ${ L}_1$ moving down stable QH subgroups (see \cite{KMel}, Section 7.3) into another system $L_2$ so that they have
the same fundamental solutions.  Then $size (TEnv(S_1; L_2))\leq size(S_1)$ for the tight enveloping fundamental sequence for $S_1=1$ constructed from the system $L_2$.
If all the parameters (Kurosh rank, size, $ab$) are the same, then
$TEnv(S_1; L_2)$ has one level, and the abelian decomposition has the same graph and QH and abelian vertex groups as $S_1$.  Notice, that  the
Kurosh rank of the tight enveloping NTQ fundamental sequence $(TEnv
(S_1; L_2))$ is the same as the maximal Kurosh rank of the corresponding
subgroup in the terminal group in the enveloping fundamental
sequence modulo free non-cyclic factors of $\langle X_2,\ldots ,X_n\rangle$. (Notice also that we do not induce the fundamental sequence by the terminal free factor of $S_1=1$ (generated by free variables of quadratic equations in $S_1=1$), we just take its image in $\Gamma _{R(L_2)}.$)

Suppose now that $H\leq \Gamma _{R(W)},$ and $(TEnv(H; S_1))$ is a tight enveloping fundamental sequence for $H$.  Then one can construct a tight enveloping fundamental sequence $(TEnv(H; L_1))$ for $H$ such that the Kurosh rank of 
$(TEnv(H; L_1))$  is bounded by the Kurosh rank of $(TEnv
(H; S_1))$  and in the case of equality one can modify the system $L_1$ into $L_2$
so that they have
the same fundamental solutions, and $(size, ab)$ for $(TEnv(H; L_2))$ is bounded  
by the $(size, ab)$ for $(TEnv(H; S_1))$.  In the case of the equality $(TEnv(H; L_2))$ has one level, and the abelian decomposition is similar to the decomposition for $(TEnv(H; S_1)).$

\begin{prop}
1) Given a fully residually $\Gamma$ group $G=\Gamma _{R(U)}$, the canonical NTQ system $W=1$ corresponding to a branch of the canonical embedding tree $T_{CE}(\Gamma _{R(U)})$ of the system $U=1$, a system of equations ${\mathcal P}=1$  with coefficients in $\Gamma _{R(W)}$ having a solution in some extension of $\Gamma _{R(W)}$, there is an algorithm for the construction of tight enveloping NTQ groups and fundamental sequences.

2) The bound in Lemma 28 from \cite{KMel} can be found effectively.
\end{prop}
\begin{proof} 1) The first algorithm can be constructed using Proposition \ref{int}, because the construction begins with the induced NTQ group $Ind(F_{R(U)})$, and this group is relatively hyperbolic as well as $\Gamma _{R(L_1)}.$  Indeed, in the construction of tight enveloping fundamental sequences and systems we have to solve the following algorithmic problems:
find intersection of conjugates of relatively quasi-convex subgroups of total relatively hyperbolic groups and conjugating elements, and 
find solution sets of quadratic systems of equations in NTQ groups (to determine the rank of a QH subgroup).  The first problem is decidable by Proposition \ref{int}, the second is solvable by Proposition \ref{relh}.

2) The bound in Lemma 28 from \cite{KMel} can be found effectively as in Theorem \ref{term}.
\end{proof}
We also consider similarly fundamental sequences of all levels  $(m,m-1)$. We now can make all the steps of the quantifier elimination procedure (to boolean combination of formulas (\ref{AE})) algorithmically.

This proves Theorem \ref{main}.

\end{document}